\documentclass[onefignum,onetabnum]{siamart171218}
\usepackage{multirow}
\usepackage[ntheorem]{empheq}
\usepackage{verbatim}
\def\myrefs#1#2{ 
{\bigskip \noindent
{\Large \bf #2}  
 \list {[\arabic{enumi}]}{\settowidth\labelwidth{[#1]}
 \leftmargin\labelwidth 
 \advance\leftmargin\labelsep
 \usecounter{enumi} }  
 \def\newblock{\hskip .11em plus .33em minus .07em}
 \sloppy\clubpenalty4000\widowpenalty4000
 \sfcode`\.=1000\relax}  }

\def\comb#1,#2,{ \left( {#1 \atop #2 } \right)  }%
\def\prodd#1,#2,#3,{ \prod_{\scriptstyle #1 \atop\scriptstyle #2 }^{#3} }%
\def\summ#1,#2,#3,{ \sum_{\scriptstyle #1 \atop\scriptstyle #2 }^{#3} }%

\newtheorem{algor}{{\sc Algorithm}}[section]
\newtheorem{tabl}{Table}[section]

\def\betab{\begin{tabbing} 
xxxx\=xxxx\=xxx\=xx\=xx\=xx\=xx\=xx\=xx\=xx\=xx\=xx\=xx\= \kill} 
\def\entab{\end{tabbing}\vspace{-0.12in}}

\newenvironment{algorithm0}{\begin{algor} \sl }{ \end{algor} } 

\newcommand{\eq}[1]{\begin{equation}\label{#1}}
\newcommand{\en}{\end{equation}}

\newcommand{\beeq}[1]{\begin{equation}\label{#1}}
\newcommand{\eneq}{\end{equation}}

\usepackage{graphicx,amsmath,amsfonts,amssymb,subfigure}
\usepackage{tikz}
\usetikzlibrary{arrows,patterns}
\usepackage{mathtools}
\usepackage{booktabs}
\usepackage{textcomp}
\usepackage{setspace}
\usepackage{enumerate}
\usepackage{enumitem}
\usepackage{hyperref}
\usepackage{pgfplots}
\usepackage{listings}
\usepackage{hyperref}
\usepackage{color} 
\definecolor{mygreen}{RGB}{28,172,0} 
\definecolor{mylilas}{RGB}{170,55,241}
\definecolor{myred2}{RGB}{236,1,59}
\definecolor{myblue2}{RGB}{116,122,255}
\definecolor{mygreen2}{RGB}{110,184,129}
\definecolor{mygreen3}{RGB}{80,229,109}
\definecolor{mygray2}{RGB}{196,196,196}
\usetikzlibrary{decorations.markings}
\usepackage{pgfplots}
\usetikzlibrary{matrix}
\pgfplotsset{compat=newest}
\usepgfplotslibrary{groupplots}
\usetikzlibrary{pgfplots.groupplots}
\usetikzlibrary{plotmarks}
\usetikzlibrary{calc}
\usepgfplotslibrary{external}
\pgfplotsset{compat=1.3}

\usetikzlibrary{shapes.geometric}
\usepackage{fancyhdr}
\usepackage{mathtools}

\usepackage{array}
\newcolumntype{R}{>{$}r<{$}}
\newcolumntype{V}[1]{>{[\;}*{#1}{R@{\;\;}}R<{\;]}}
\usetikzlibrary{plotmarks}
\usepackage{pgf}
\usepackage[utf8]{inputenc}
\usetikzlibrary{arrows,automata}
\usetikzlibrary{positioning}

\tikzset{
    state/.style={
           rectangle,
           rounded corners,
           draw=black, very thick,
           minimum height=2em,
           inner sep=2pt,
           text centered,
           },
}

\graphicspath{{./FIGS/}}

\newcommand{\xpmatrix}[1]{\begin{pmatrix} #1 \end{pmatrix}}

\newcommand{\ra}[1]{\renewcommand{\arraystretch}{1}}

\newcommand{\TheTitle}{Fast randomized non-Hermitian eigensolvers based on rational filtering and matrix partitioning} 
\newcommand{\TheShortTitle}{ Rational filtering non-Hermitian eigensolvers} 

\newcommand{\TheAuthors}{Vassilis Kalantzis, Yuanzhe Xi, and Lior Horesh}
\headers{\TheShortTitle}{\TheAuthors}
\author{
    Vassilis Kalantzis\thanks{IBM Research, Thomas J.\ Watson Research Center, Yorktown Heights, NY 10598, USA
    (\email{vkal@ibm.com}, \email{lhoresh@us.ibm.com}).}
  \and
  Yuanzhe Xi\thanks{Department of Mathematics, Emory University, Atlanta, GA 30322, USA
    (\email{yxi26@emory.edu}).}
  \and
  Lior Horesh\footnotemark[2]}

\headers{\TheShortTitle}{\TheAuthors}

\title{{\TheTitle}}

\begin{document}

\maketitle

\begin{abstract}
This paper describes a set of rational filtering algorithms to compute a few eigenvalues 
(and associated eigenvectors) of non-Hermitian matrix pencils. Our interest lies in 
computing eigenvalues located inside a given disk, and the proposed algorithms approximate 
these eigenvalues and associated eigenvectors by harmonic Rayleigh-Ritz projections on 
subspaces built by computing range spaces of rational matrix functions through randomized 
range finders. These rational matrix functions are designed so that directions 
associated with non-sought eigenvalues are dampened to (approximately) zero. Variants 
based on matrix partitionings are introduced to further reduce the overall complexity of 
the proposed framework. Compared with existing eigenvalue solvers based on rational 
matrix functions, the proposed technique requires no estimation of the number of 
eigenvalues located inside the disk. Several theoretical and practical issues are 
discussed, and the competitiveness of the proposed framework is demonstrated via 
numerical experiments.
\end{abstract}



\section{Introduction}

This paper describes a rational filtering framework to compute a few eigenvalues 
and associated eigenvectors of non-Hermitian eigenvalue problems of the form 
\begin{equation}
Ax=\lambda Mx,  
\label{eq:problem}
\end{equation}
where the matrices $A\in\mathbb{C}^{n\times n}$ and $M\in\mathbb{C}^{n\times n}$ 
are assumed large and sparse, and the pencil $(A,M)$ is assumed regular and diagonalizable. The focus of this paper lies in computing all eigenvalues 
located in the interior of a disk ${\cal D}$ prescribed in the complex domain.
An illustrative example is shown in Figure \ref{fig:example}.
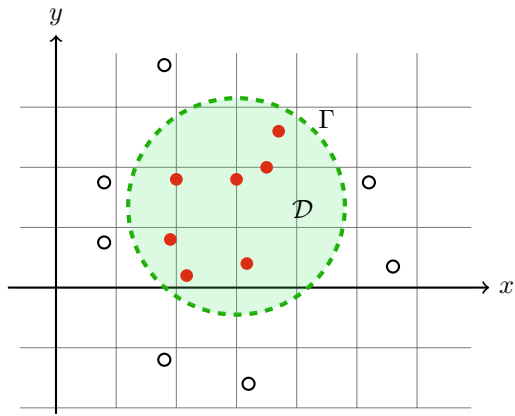
\begin{figure}
\centering
\begin{tikzpicture}[domain=0:4,scale = 0.8]
  \draw[color=gray] (-0.6,-2.0) grid (6.9,3.9);
  \draw[->,thick] (-.8,0) -- (7.2,0) node[right] {$x$};
  \draw[->,thick] (0,-2.1) -- (0,4.2) node[above] {$y$};
  \coordinate (a) at (2.1712,0.2);
  \fill[color=red] (a) circle (3pt);
  \coordinate (a) at (3.1712,0.4);
  \fill[color=red] (a) circle (3pt);
  \coordinate (a) at (1.9,0.8);
  \fill[color=red] (a) circle (3pt);
    \coordinate (a) at (3.7,2.6);
  \fill[color=red] (a) circle (3pt);
    \coordinate (a) at (3.5,2.0);
  \fill[color=red] (a) circle (3pt);
    \coordinate (a) at (3.0,1.8);
  \fill[color=red] (a) circle (3pt);
  \coordinate (a) at (2.0,1.8);
  \fill[color=red] (a) circle (3pt);
  %
  \draw [dashed,mygreen,ultra thick] (3,1.35) circle [radius=1.8];
  \fill[color=mygreen3, opacity=0.2] (3,1.35) circle [radius=1.8];
  \node (Lend) at (4.1,1.3) {${\cal D}$};
  \node (Lend) at (4.5,2.8) {$\Gamma$};
  %
        \draw [solid, thick] (5.6,0.35) circle [radius=0.1];
        \draw [solid, thick] (5.2,1.75) circle [radius=0.1];
        \draw [solid, thick] (0.8,0.75) circle [radius=0.1];
        \draw [solid, thick] (0.8,1.75) circle [radius=0.1];
        \draw [solid, thick] (1.8,3.7) circle [radius=0.1];
        \draw [solid, thick] (3.2,-1.6) circle [radius=0.1];
        \draw [solid, thick] (1.8,-1.2) circle [radius=0.1];
    \end{tikzpicture}
\caption{Sought eigenvalues are denoted by red filled dots. 
Unwanted eigenvalues located outside the circumference 
$\Gamma$ (denoted by a green dashed curve) of the disk 
${\cal D}$ are denoted by black solid circles. 
\label{fig:example}}
\end{figure}

Rational filtering eigenvalue solvers can be seen as (harmonic) Rayleigh-Ritz 
procedures in which the projection subspace is built by exploiting 
a (complex) rational transformation of the matrix pencil 
$(A,M)$. These transformations are constructed so that the gap 
between eigenvalues located inside the disk ${\cal D}$ versus those 
located outside the latter is as big as possible after the transformation. 
Applying a projection scheme to the transformed pencil can 
then significantly enhance the convergence towards the sought invariant 
subspace. The most popular approaches to construct efficient rational 
transformations is either via shift-and-invert or via a discretization 
of the Cauchy integral representation of the eigenprojector along the 
boundary of the disk ${\cal D}$
\cite{beyn2012integral,FEAST,sakurai2003projection,sakurai2007cirr,ye2017fast,yin2017feast}.
Compared to shift-and-invert, contour integral eigensolvers 
are generally oblivious to the location of the sought eigenvalues inside 
the disk ${\cal D}$, and enjoy enhanced scalability when implemented in distributed 
memory computing environments \cite{aktulga2014parallel,iwase2017efficient,ddfeast,pfeast,YAMAZAKI2013280}. 
Other rational filters, though not necessarily based on contour integration, 
can be found in 
\cite{doi:10.1137/140984129,VANBAREL2016346,guttel2015zolotarev,kollnig2020rational,doi:10.1137/18M1170935,doi:10.1137/18M1228128,ruhe1998rational,winkelmann2017non,xi2016computing,doi:10.1137/16M1078409}. 

{
In this paper we consider algorithms in which the projection subspace 
is set equal to the column space of matrices formed after applying a 
rational transformation to the matrix pencil $(A,M)$. These rational 
transformations are constructed so that eigenvector directions associated 
with eigenvalues located outside the disk ${\cal D}$ are approximately 
mapped to zero, and the corresponding column spaces are captured through 
randomized range finders \cite{martinsson2019randomized,martinsson2020randomized}. } 
The algorithms proposed in this paper are also 
combined with matrix partitionings to reduce the computational complexity 
of the construction of the projection subspace. So far matrix partitioning 
approaches have been featured within the context of rational filtering 
only for symmetric eigenvalue problems 
\cite{ddfeast,kalantzis2018beyond}. One of the main motivations of this 
paper is to extend this class of techniques to non-Hermitian eigenvalue 
problems.  

Overall, the proposed framework possesses the following advantages:

\textbf{Improved robustness}. Classical rational filtering approaches such as
FEAST \cite{FEAST} or the SS algorithm \cite{sakurai2003projection,sakurai2007cirr} require 
an estimation of the number of eigenvalues located inside the disk ${\cal D}$. 
However, such an estimation is not always readily available or easy to compute for 
generalized eigenvalue problems. On the other hand, an inaccurate estimation can lead to 
slow convergence or failure to capture all required eigenpairs. The proposed algorithms 
bypass this issue by dynamically increasing the dimension of the projection subspace.

\textbf{Reduced complexity}. By combining rational filtering with substructuring, the projection subspace is formed as the direct sum of two separate subspaces approximated independently. This is done by applying a $2\times 2$ block partitioning to the pencil $(A,M)$. Two specialized algorithms are proposed to further reduce the 
computational costs associated with classical rational filtering eigensolvers.
The $2\times 2$ block partitioning can be created either in an ad hoc way or by applying a graph partitioner to the adjacency graph of the pencil $(A,M)$.

\textbf{Enhanced parallelism}. In addition to the ample opportunities for parallelism 
offered by rational filtering eigensolvers, the proposed algorithms can take advantage of an 
additional level of parallelism introduced by matrix substructuring.

The structure of this paper is organized as follows. Section \ref{sec2} describes 
a technique based on the combination of randomized range finders, harmonic Rayleigh-Ritz projections and rational transformations. Section \ref{sec:3} presents 
two variants based on matrix partitioning which aim at reducing the computational 
cost associated with the construction of an efficient projection subspace. Section \ref{sec4a} discusses practical details and presents computational cost comparisons.
Section \ref{sec4} provides numerical experiments on a few test problems. Finally, 
Section \ref{sec5} presents our concluding remarks.

\subsection{Notation} 

Throughout this paper we denote the spectrum of $(A,M)$ by $\Lambda(A,M)$. 
The total number of eigenvalues located inside the disk ${\cal D}$ is assumed 
unknown and is denoted by $n_{ev}$. The eigentriplets of the matrix pencil 
$(A,M)$ are denoted by $\left(\lambda_i,x^{(i)},\hat{x}^{(i)}\right),\ i=1,\ldots,n$, 
where $\lambda_i$ denotes the $i$th eigenvalue of smallest distance from 
the center of the disk ${\cal D}$, and $x^{(i)}$ and $\left(\hat{x}^{(i)}\right)^H$ 
denote the corresponding right and left eigenvectors, respectively. 
Notice that using the above definition we have $\lambda_1,\ldots,\lambda_{n_{ev}} 
\in {\cal D}$ and $\lambda_{n_{ev}+1},\ldots,\lambda_n \notin {\cal D}$.
The superscript “$H$” denotes the conjugate transpose of the corresponding matrix. 
Unless mentioned otherwise, the term “eigenvector" should be understood to 
refer to a right eigenvector. Throughout the rest of this paper we use 
the notation $\texttt{rank}(X)$, $\texttt{orth}(X)$, and $\texttt{range}(X)$ 
to denote the rank, orthonormalization, and range (column space) of the $m\times n$ 
matrix $X$, respectively. Moreover, we use the 
notation $\texttt{span}\left(r^{(1)},\ldots,r^{(k)}\right)$ to denote the linear 
span of vectors $r^{(1)},\ldots,r^{(k)}$. 

\section{Harmonic Rayleigh-Ritz projections and randomized range finders for column spaces of matrix functions} 
\label{sec2}

Computing a few exterior eigenvalues and associated eigenvectors 
of large and sparse matrix pencils is typically achieved via 
applying a Rayleigh-Ritz procedure (RR) onto a (nearly) invariant 
subspace associated with the sought eigenvalues \cite{parlett1998symmetric}.
For Hermitian eigenvalue problems, the RR procedure retains several 
optimality properties, e.g., see \cite{li2015rayleigh}. For non-Hermitian 
eigenvalue problems no such optimality is guaranteed, e.g., when the 
sought eigenvalues are located in the interior of the spectrum, 
e.g., inside a disk ${\cal D}$ surrounded by several unwanted 
eigenvalues, the RR procedure might provide poor results \cite{morgan1998harmonic}. 

An alternative for the solution of interior eigenvalue problems is 
the harmonic Rayleigh-Ritz procedure (HRR) suggested in 
\cite{morgan1991computing}. More specifically, let matrix $Z$ 
represent a basis of some projection subspace ${\cal Z}$. The HRR 
procedure extracts approximate eigenpairs of the form $(\theta,Zq)$ 
by solving the following eigenvalue problem 
\begin{equation}\label{hrr}
    Z^H(A-\zeta_c M)^H(A-\zeta_c M)Zq = (\theta-\zeta_c)Z^H
    (A-\zeta_c M)^HMZq,\ \ \ \zeta_c \in \mathbb{C}.
\end{equation}
For eigenvalue problems such as the ones considered in this paper, it is 
reasonable to set $\zeta_c$ equal to the center of the disk ${\cal D}$. 
The approximate eigenvalue $\theta$ and eigenvector $Zq$ are referred to 
as (harmonic) Ritz value and Ritz vector, respectively. 
In practice, if the subspace ${\cal Z}$ includes the sought invariant 
subspace, then (\ref{hrr}) will return accurate approximations of the 
corresponding eigenpairs provided that there are no spurious eigenvalues 
close to the Ritz values located 
inside the disk ${\cal D}$ \cite{jia2005convergence,morgan1991computing}.

Based on the above discussion, we seek to compute a subspace ${\cal Z}$ 
which includes the invariant subspace associated with $n_{ev}$ sought 
eigenvalues $\lambda_1,\ldots,\lambda_{n_{ev}}$. This section considers 
ansatz subspaces of the form ${\cal Z}=\texttt{range} \left(\rho(M^{-1}A)\right)$ 
for some scalar function $\rho$ such that $\rho(M^{-1}A)$ is rank-deficient. 
The rest of this section considers such a function $\rho$ while it also 
discusses a randomized algorithm to compute the range of rank-deficient 
matrices. 

\subsection{Fast randomized range finder for rank-deficient matrices}

{Let $X\in \mathbb{C}^{m\times n}$ be a rectangular matrix. 
The goal of a range finding procedure is to compute an orthonormal matrix $Y$ 
such that $\|(I-YY^H)X\|$ is zero. In this paper we are interested in scenarios 
where matrix $X$ is rank-deficient and accessible only through a Matrix-Vector 
product routine.}

{Let $k\in \mathbb{N},\ k< {\rm min}(m,n)$, denote the 
rank of matrix $X$. The range of matrix $X$ is equal to the span of 
the left-singular vectors corresponding to the $k$ non-zero singular values. 
The span of these left-singular vectors can be computed in a matrix-free 
fashion by Lanczos bidiagonalization (LBD) \cite{golub1965calculating}. In 
the absence of round-off errors, LBD requires $k$ Matrix-Vector products with 
each of the matrices $X$ and $X^H$, in addition to the cost introduced by 
the chosen orthogonalization strategy, e.g., see \cite{hernandez2007restarted}. 
Alternatively, we can apply $k$ steps of the Lanczos process on $XX^H$ 
\cite{lanczos1950iteration}, but this approach still requires $2k$ Matrix-Vector 
products overall.}

\vspace{0.1in}
\noindent\fbox{%
\begin{minipage}{\dimexpr\linewidth-9\fboxsep-9\fboxrule\relax}
\vbox{
\begin{algorithm0}{Randomized range finding algorithm} \label{alg:2}
\betab
\>0.\> Inputs: $X\in \mathbb{C}^{m\times n},\ Y:=0$ \\
\>1a.\> For $i=1,\ldots,{\rm min}(m,n)$\\
\>2.\> \> Fill $r\in \mathbb{C}^n$ with normally distributed random entries\\
\>3.\> \> $Y=[Y,Xr]$\\
\>4.\>\> Set the $i \times 1$ vector $\sigma^{(Y)}$ equal to the (sorted) singular 
\\ \>\>\> values of matrix $Y$\\
\>5.\> \> If $\sigma_i^{(Y)}/\sigma_1^{(Y)} \leq {\rm machine\ epsilon}$, break; \\
\>1b.\> End\\
\>6.\> Orthonormalize and return $Y$
\entab
\end{algorithm0}
}\vspace{0.05in}
\end{minipage}
}
\vspace{0.05in}

{
The complexity of the range finding problem can be reduced by considering 
techniques from randomized linear algebra \cite{drineas2016randnla,liberty2007randomized,martinsson2020randomized}. 
Randomized numerical algorithms have gained significant prominence over 
the last two decades due to their superior performance in several important 
numerical linear algebra problems, e.g., low-rank matrix approximations \cite{halko2011finding,mahoney2009cur} 
and principal component analysis \cite{bose2019terapca,rokhlin2010randomized}. 
Returning to the range finding problem, let $R\in \mathbb{C}^{n\times k}$ be 
a matrix whose entries are drawn from a Gaussian distribution. Then, with 
probability one, we have  
$\texttt{rank}(XR)=k$ and $\texttt{range}(XR)=\texttt{range}(X)$ 
\cite{martinsson2019randomized}. Thus, a randomized range finder requires only 
half of the Matrix-Vector products performed by LBD or Lanczos. Our interest 
lies in scenarios where the exact rank of matrix $X$ is either unknown or 
expensive to estimate. To bypass this issue, next we consider a modification 
of the randomized range finder where the Matrix-Vector products with matrix $X$ 
are performed in an incremental manner and no information regarding $k$ is 
needed.}

{
Let $r^{(i)}\in \mathbb{C}^n,\ i=1,2,\ldots$, denote a sequence of vectors with 
normally randomly distributed entries, and $\left[Xr^{(1)},Xr^{(2)},\ldots\right]$ 
denote the evolving matrix in which we accumulate the products of matrix $X$ with 
$r^{(i)}$. After $k$ such products, the rank (and range) of the evolving matrix 
is equal to that of matrix $X$. Since the following Matrix-Vector products 
$Xr^{(i)},\ i=k+1,k+2,\ldots$, already lie in $\mathtt{range}(X)$, the evolving 
matrix will become singular. Therefore, we can bypass the unknown rank of matrix 
$X$ by monitoring the singular values of the evolving matrix. The above approach 
is listed as Algorithm \ref{alg:2}. The procedure terminates when the ratio of 
the smallest to the largest singular value of the matrix $[Xr^{(1)},Xr^{(2)},\ldots]$ 
becomes zero, which in a numerical computing environment translates to smaller 
than or equal to the machine epsilon.\footnote{We consider a 
number to be equal to zero if its numerical value is less than the machine 
epsilon of the IEEE 754 binary64 definition.} In the absence of round-off errors, 
Algorithm \ref{alg:2} terminates after $k+1$ iterations. The SVD of the evolving 
matrix in Step 3 can be updated on-the-fly each time a new column is added \cite{kalantzis2020projection,zha1999updating}. We note here that Algorithm 
\ref{alg:2} can be also seen as a variation of the adaptive range finder described 
in \cite[Section 4]{halko2011finding} with the exception that the stopping 
criterion is based on the magnitude of the condition number of the evolving matrix. }


{
While our interest lies in computing the exact $\mathtt{range}(X)$, in practice 
we stop the iterative procedure in Algorithm \ref{alg:2} when the ratio of 
the smallest to the largest singular value becomes less than a small threshold, 
e.g., $10^{-12}$. This helps to avoid orthonormalizing an ill-conditioned basis 
at the last step of Algorithm \ref{alg:2}, e.g., see \cite{giraud2002modified}. 
If Algorithm \ref{alg:2} terminates after $k_0+k_1<k+1$ iterations, then, with 
probability at least $1-6k_{1}^{-k_{1}}$, the matrix $Y = \texttt{orth}(X\left[r^{(1)},\ldots,r^{(k_0+k_1)}\right])$ satisfies 
\cite{halko2011finding}:}
\begin{equation} \label{martinsson}
    \|(I-YY^H)X \|_2 \leq \left(1+11\sqrt{k_0+k_1}\sqrt{{\rm min}(m,n)}\right)
    \sigma_{k_0+1}(X),
\end{equation}
{where $\sigma_j(X)$ denotes the $j$th singular value of matrix $X$. 
Therefore, when the singular values of matrix $X$ decay fast enough, Algorithm 
\ref{alg:2} can still return a good approximation of $\mathtt{range}(X)$ in 
less than $k+1$ iterations. Note though that we can not predict predict the number 
of iterations associated with a higher stop tolerance in Algorithm \ref{alg:2}. }

\subsection{Column spaces of matrix functions as projection subspaces}

Let $\rho: \mathbb{C}^*\rightarrow \mathbb{R},\ \mathbb{C}^* \subseteq \mathbb{C}$, 
be a scalar function that is defined over $\Lambda(A,M)$. Since $(A,M)$ is 
diagonalizable, applying the function $\rho$ to matrix $M^{-1}A$ is equivalent to
\begin{equation}\label{eq:sum0}
\rho(M^{-1}A)  = \sum_{i=1}^{n} \rho(\lambda_i)x^{(i)} 
\left(\hat{x}^{(i)}\right)^H M.
\end{equation}
Notice now that $\texttt{span}\left(x^{(1)},\ldots,x^{(n_{ev})}\right)
\subseteq \texttt{range}(\rho(M^{-1}A))$ for any function $\rho$ such that 
$\rho(\lambda_i) \neq 0,\ i=1,\ldots,n_{ev}$. 
Algorithm \ref{alg:00} outlines a two-step procedure to approximate the 
eigenvalues located inside the disk ${\cal D}$ and associated eigenvectors. 
The first step is to compute an orthonormal basis matrix $Z$ of 
$\texttt{range}(\rho(M^{-1}A))$ by calling Algorithm \ref{alg:2}. 
{The number of iterations performed by Algorithm 
\ref{alg:2} is bounded by the number of eigenvalues $\lambda$ that 
satisfy $\rho(\lambda) \neq 0$.} Therefore, the scalar function $\rho$ 
should be set such that $\rho(\lambda_i)$ is about equal to machine precision 
for as many eigenvalues $\lambda_i \notin {\cal D}$ as possible. 
The second step is to perform a HRR projection step to approximate 
the eigenvalues located inside ${\cal D}$ and their associated 
eigenvectors. Note that no information about the value of $n_{ev}$ 
is required.

\vspace{0.1in}
\noindent\fbox{%
\begin{minipage}{\dimexpr\linewidth-9\fboxsep-9\fboxrule\relax}
\vbox{
\begin{algorithm0}{Prototype algorithm} \label{alg:00}
\betab
\>0.\> Inputs: $\rho: \mathbb{C} \rightarrow \mathbb{R},\ {\cal D}$\\
\>1.\> Compute an orthonormal basis $Z$ of $\mathtt{range}(\rho(M^{-1}A))$ \\ 
\> \> by Algorithm \ref{alg:2}\\
\>2.\> Solve the eigenvalue problem in (\ref{hrr}) and return all \\\>\>
Ritz values $\theta \in {\cal D}$ and associated Ritz vectors
\entab
\end{algorithm0}
}\vspace{0.05in}
\end{minipage}
}
\vspace{0.1in}

Motivated by the above discussion, an ideal function $\rho$ is defined by the 
contour integral
\begin{equation} \label{eq:contour}
{\cal P}(\zeta) = \dfrac{-1}{2\pi i}\int_{\Gamma}\dfrac{1}{\zeta-\nu} d\nu,
\end{equation}
where the complex contour $\Gamma$ denotes the circumference of the disk ${\cal D}$, 
and the integration is performed counter-clockwise. By Cauchy's residue theorem it 
follows that ${\cal P}(\zeta)=1$ for any $\zeta \in {\cal D}$, and zero otherwise. 
Applying (\ref{eq:contour}) to (\ref{eq:sum0}) yields 
\begin{equation} \label{eqqq1}
\begin{aligned}
{\cal P}(M^{-1}A) & = \dfrac{-1}{2\pi i}\int_{\Gamma} \left(M^{-1}A-\nu I\right)^{-1} d\nu 
= \sum_{i=1}^{n_{ev}} x^{(i)} \left(\hat{x}^{(i)}\right)^H M.
\end{aligned}
\end{equation}
Algorithm \ref{alg:00} then terminates after exactly $n_{ev}$ iterations. 

In practice, \eqref{eq:contour} will be approximated by numerical quadrature 
which leads to a rational “filter" function of the form
\begin{equation} \label{ratfilter}
\rho(\zeta) = \sum_{j=1}^N \dfrac{\omega_j}{\zeta-\zeta_j},
\end{equation}
where the integer $N$ denotes the order of the approximation, and the complex 
pairs $\{\omega_j,\zeta_j\}_{j=1,\ldots,N}$ denote the weights and nodes of 
the quadrature rule, respectively. Rational filter functions of 
this form were pioneered in the context of eigenvalue solvers first 
in \cite{asakura2009numerical,beyn2012integral,FEAST,sakurai2003projection}. 
The application of (\ref{ratfilter}) to the pencil $(A,M)$ then gives
\begin{equation} \label{eq:sum}
\rho(M^{-1}A) = \sum_{j=1}^{N} \omega_j (M^{-1}A-\zeta_j I)^{-1} 
= \sum_{j=1}^{N} \omega_j (A-\zeta_j M)^{-1} M, 
\end{equation}
and computing $\rho(M^{-1}A)r$ for a vector $r$ at each iteration of Algorithm 
\ref{alg:00} involves: a) one Matrix-Vector product  
with matrix $M$, and b) the solution of one linear system with 
each matrix $A-\zeta_j M,\ j=1,\ldots,N$. 
These $N$ linear system solutions can be obtained in parallel 
by replicating matrices $A$ and $M$ in $N$ different groups of 
processors. 

Ideally, the function in (\ref{ratfilter}) should decay to zero as 
$\zeta$ moves away from ${\cal D}$. Figure 
\ref{filtermag} plots the modulus of a rational filter $\rho(\zeta)$ 
defined on the unit disk (${\cal D}\equiv \{|z|:|z|\leq 1\}$) with the 
trapezoidal rule of order $N=8$ (left) and $N=16$ (right). Increasing 
the value of $N$ leads to a faster decay of the rational filter 
$\rho(\zeta)$ outside the boundary of ${\cal D}$. In particular, the 
approximation of ${\cal P}(\zeta)$ by $\rho(\zeta)$ at the center of 
the disk ${\cal D}$ converges exponentially\footnote{Note that eigenvalues 
$\lambda$ located very close to the poles $\zeta_j$ can lead to values 
$\rho(\lambda)$ which are larger than one even if $\lambda \notin {\cal D}$.} 
with respect to $N$ \cite{doi:10.1137/130931035,trefethen2014exponentially}. 
\begin{figure} 
\centering
\includegraphics[width=0.49\textwidth]{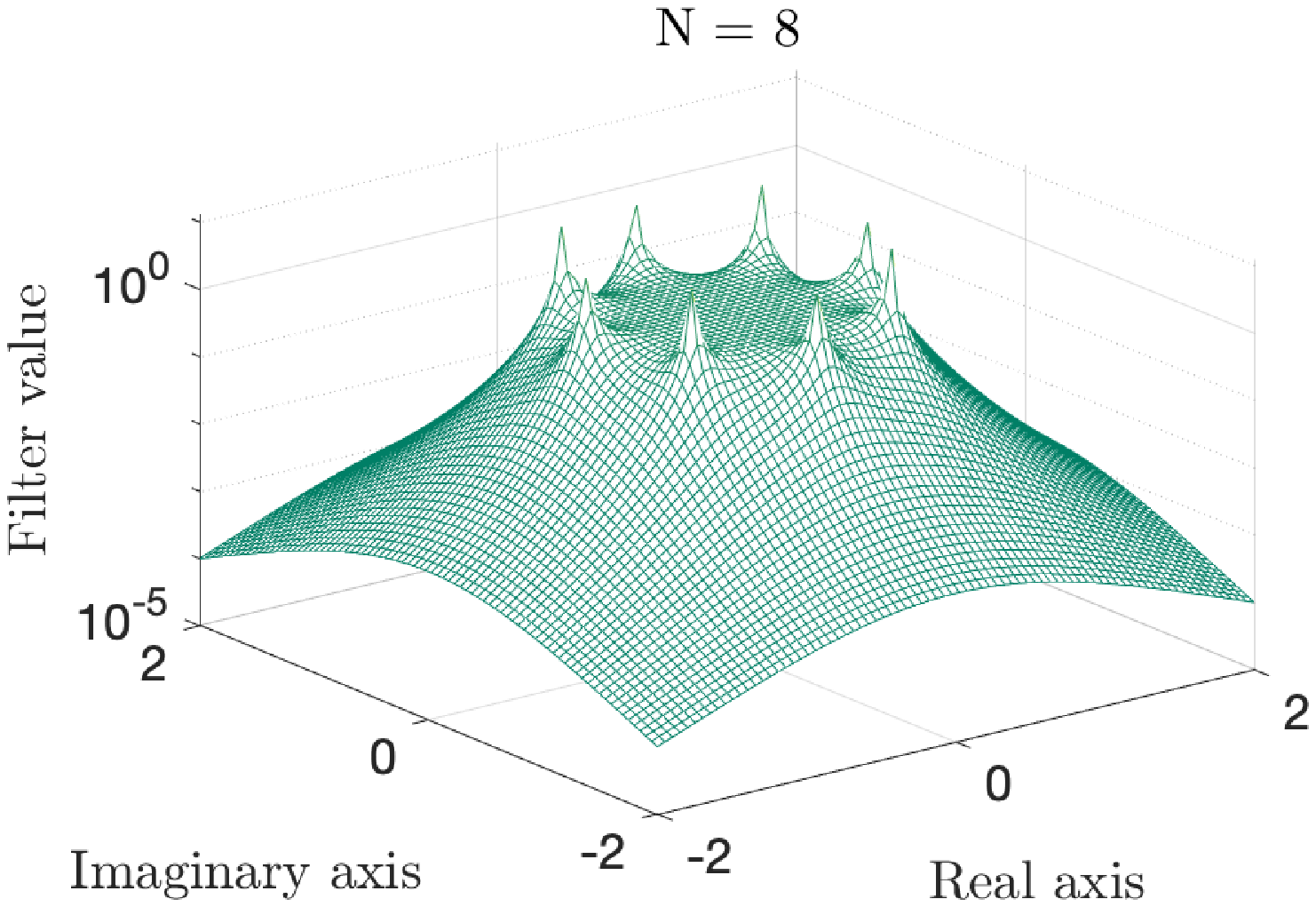}
\includegraphics[width=0.49\textwidth]{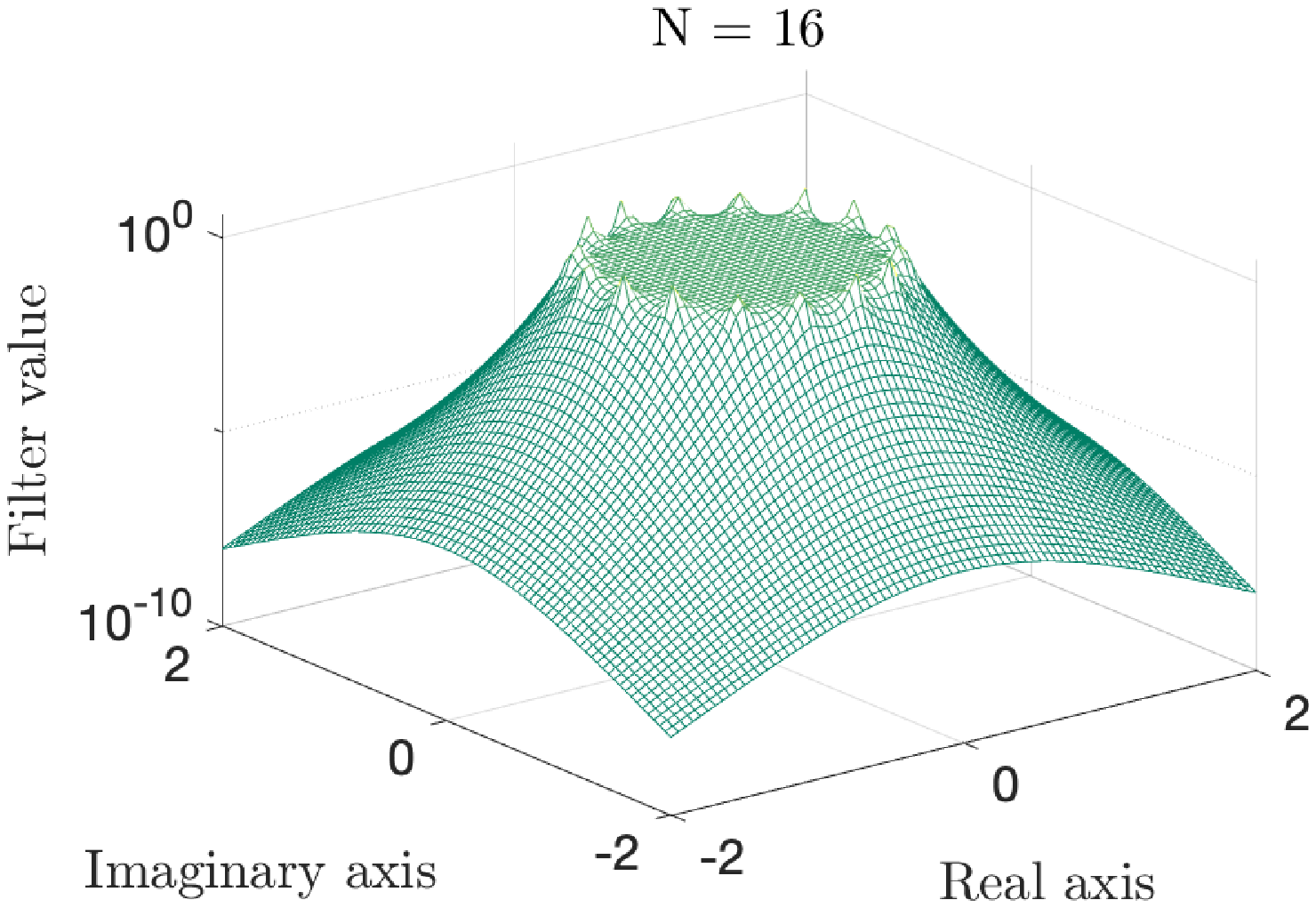}
\caption{The modulus of the rational filter $\rho(\zeta)$ defined 
on the unit disk with the trapezoidal rule of order $N=8$ (left) and 
$N=16$ (right). \label{filtermag}} 
\end{figure}

The convergence of Algorithm \ref{alg:2} is likely to be slow for small 
values of $N$ since $\texttt{range}(\rho(M^{-1}A))$ could 
contain many eigenvector directions associated with a large number of eigenvalues located outside 
${\cal D}$. As a result, subspace iteration might be a better alternative 
in this case, and this is exploited in the FEAST eigenvalue solver library 
\cite{peter2014feast,FEAST}. The drawback of subspace iteration as a 
projection scheme is that a good estimation of $n_{ev}$ is necessary (e.g., 
see \cite{di2016efficient,Yin2018,yin2019harmonic}), a condition which is 
bypassed by Algorithm \ref{alg:00}.

Throughout the rest of this paper we describe two variants which aim at 
reducing the computational cost of Algorithm \ref{alg:00}.

\section{Algorithms based on matrix partitionings} \label{sec:3}

Let $d,\ s \in \mathbb{N}$, such that $n=s+d$, and partition each eigenvector 
$x^{(i)},\ i=1,\ldots,n,$ of the pencil $(A,M)$ as
\begin{equation} \label{eig_part}
\ \ \ \ \ \ 
x^{(i)}=
 \begin{pmatrix}
   u^{(i)}                                     \\[0.3em]
   y^{(i)}                                     \\[0.3em]
  \end{pmatrix},\ u^{(i)} \in \mathbb{C}^{d},\ y^{(i)}\in \mathbb{C}^s.
\end{equation}
In addition, let $0_{\chi,\psi}$ denote the zero matrix of size $\chi \times \psi$. 
Then, we can write
\begin{align} \label{ideal}
 \texttt{span} \left(x^{(1)},\ldots,x^{(n_{ev})} \right) 
 & =
 \texttt{span} 
 \left( \begin{bmatrix}
 u^{(1)},\ldots,u^{({n_{ev}})} \\[0.3em]
 0_{s,n_{ev}} \\[0.3em]
 \end{bmatrix}+
 \begin{bmatrix}
 0_{d,n_{ev}} \\[0.3em]
 y^{(1)},\ldots,y^{({n_{ev}})} \\[0.3em]
 \end{bmatrix}
 \right) \\
 & \subseteq
 \texttt{span} 
 \left( \begin{bmatrix}
 u^{(1)},\ldots,u^{({n_{ev}})} \\[0.3em]
 0_{s,n_{ev}} \\[0.3em]
 \end{bmatrix}\right) \oplus 
 \texttt{span} \left(\begin{bmatrix}
 0_{d,n_{ev}} \\[0.3em]
 y^{(1)},\ldots,y^{({n_{ev}})} \\[0.3em]
 \end{bmatrix}
 \right).
 \end{align}
The expression in (\ref{ideal}) implies that $\texttt{span} \left(x^{(1)},\ldots,x^{(n_{ev})} \right)$ is captured by the direct sum of 
$\texttt{span} \left(u^{(1)},\ldots,u^{(n_{ev})}\right)$ and 
$\texttt{span} \left(y^{(1)},\ldots,y^{(n_{ev})} \right)$. 
The rest of this section describes two variations of Algorithm 
\ref{alg:00}, presented in Sections \ref{first_alg} and \ref{second_alg}. 
These algorithms make use of matrix partitioning to reduce the computational 
costs associated with the construction of a good HRR projection subspace.

\subsection{Two equivalent matrix resolvent representations}

Consider the following $2\times 2$ block-partitioning of the non-Hermitian 
matrices $A$ and $M$:
\begin{equation}\label{eq1}
A = 
\begin{pmatrix} 
B     &  F \cr    E   &  C 
\end{pmatrix} \quad\text{and}\quad
M = 
\begin{pmatrix} 
M_B   &  M_F \cr  M_E & M_C 
\end{pmatrix},
\end{equation}
where $B,\ M_B \in \mathbb{C}^{d\times d},\ F,\ M_F \in \mathbb{C}^{d\times s},\ 
E,\ M_E \in \mathbb{C}^{s\times d}$, and $C,\ M_C \in \mathbb{C}^{s\times s}$.
Moreover, define the following matrix-valued functions of 
$\zeta \in \mathbb{C}$:
\begin{equation*}\label{def0}
 B(\zeta) =B-\zeta M_B,\ \ F(\zeta) =F-\zeta M_F,\ \ E(\zeta) =E-\zeta M_E,\ {\rm and} \ \ 
 C(\zeta) =C-\zeta M_C.
\end{equation*}

For any $\zeta \notin \Lambda(A,M)$, the matrix $(A-\zeta M)^{-1}$ can be written as
{\normalsize \begin{equation} \label{inv_dd2}
  (A-\zeta M)^{-1}=
  \begin{pmatrix}
   B(\zeta)^{-1}\left[ I + F(\zeta)S(\zeta)^{-1}E(\zeta)B(\zeta)^{-1}\right]   &   -B(\zeta)^{-1}F(\zeta)S(\zeta)^{-1}    \\[0.3em]
   -S(\zeta)^{-1}E(\zeta)B(\zeta)^{-1}       &   S(\zeta)^{-1}                                           \\[0.3em]
  \end{pmatrix},
\end{equation}}
where the $s\times s$ matrix-valued function
\begin{equation*}
S(\zeta) = C(\zeta) -E(\zeta)B(\zeta)^{-1}F(\zeta)
\end{equation*}
is the \emph{Schur complement} of matrix $A-\zeta M$. 
Combining (\ref{inv_dd2}) with (\ref{eq:sum}) then gives
\begin{equation}\label{num_cont_big}
{\footnotesize  
\rho(M^{-1}A) =\sum_{j=1}^{N} \omega_j
\begin{bmatrix}
   B(\zeta_j)^{-1}\left[I + F(\zeta_j)S(\zeta_j)^{-1}E(\zeta_j)B(\zeta_j)^{-1}\right] & -  B(\zeta_j)^{-1}F(\zeta_j)S(\zeta_j)^{-1} \\[0.3em]
 - S(\zeta_j)^{-1}E(\zeta_j)B(\zeta_j)^{-1}                                                            &  S(\zeta_j)^{-1}   \\[0.3em]
\end{bmatrix} M.}
\end{equation}

Similarly, the matrix $(A-\zeta_j M)^{-1}$ can be expressed 
in terms of the eigenvectors of the matrix pencil $(A,M)$ as
\begin{equation} \label{pole_eig} 
\begin{aligned}
(A-\zeta_j M)^{-1} 
                & = \sum_{i=1}^n \dfrac{x^{(i)} \left(\hat{x}^{(i)}\right)^H}{\lambda_i-\zeta_j}. 
\end{aligned}
\end{equation}
Then, by partitioning the left eigenvectors of the pencil $(A,M)$ 
as in (\ref{eig_part}),
\begin{equation*} \label{eig_part2}
  \left(\hat{x}^{(i)}\right)^H=
 \begin{bmatrix}
   \left(\hat{u}^{(i)}\right)^H & \left(\hat{y}^{(i)}\right)^H
  \end{bmatrix},\ \left(\hat{u}^{(i)}\right)^H \in \mathbb{C}^{1\times d},\ \left(\hat{y}^{(i)}\right)^H \in \mathbb{C}^{1\times s},
\end{equation*}
and combining (\ref{eq:sum}) with (\ref{pole_eig}), we obtain the following 
identity:
\begin{equation}\label{lala}
\rho(M^{-1}A) = \sum_{i=1}^n \rho(\lambda_i)
\begin{bmatrix}
   u^{(i)} \left(\hat{u}^{(i)}\right)^H  &  u^{(i)} \left(\hat{y}^{(i)}\right)^H   \\[0.3em] 
   y^{(i)} \left(\hat{u}^{(i)}\right)^H &  y^{(i)} \left(\hat{y}^{(i)}\right)^H
\end{bmatrix}M.
\end{equation}

\subsection{First algorithm} \label{first_alg}

In this section we present an algorithm which exploits the equivalent 
representations of $\rho(M^{-1}A)$ shown in \eqref{num_cont_big} and 
\eqref{lala} to build a subspace which captures 
$\texttt{span} \left(u^{(1)},\ldots,u^{(n_{ev})} \right)$ and 
$\texttt{span} \left(y^{(1)},\ldots,y^{(n_{ev})} \right)$. 

Equating the (1,2) and (2,2) blocks on the right-hand sides of (\ref{num_cont_big}) 
and (\ref{lala}) gives
\begin{equation} \label{eq:nb1}
\begin{aligned}
 -\sum_{j=1}^{N} \omega_jB(\zeta_j)^{-1}F(\zeta_j)S(\zeta_j)^{-1}  &= 
 \sum_{i=1}^{n} \rho(\lambda_i) u^{(i)} \left(\hat{y}^{(i)}\right)^H,\\
 \sum_{j=1}^{N} \omega_jS(\zeta_j)^{-1}  &= 
 \sum_{i=1}^{n} \rho(\lambda_i) y^{(i)} \left(\hat{y}^{(i)}\right)^H.
 \end{aligned}
\end{equation}
These identities indicate that, under mild conditions, we can capture 
a superset of $\texttt{span} \left(u^{(1)},\ldots,u^{(n_{ev})} \right)$ and 
$\texttt{span} \left(y^{(1)},\ldots,y^{(n_{ev})} \right)$ by capturing the 
range of the matrices on the left-hand side in (\ref{eq:nb1}). 
\begin{theorem}
\label{pro2}
Let $\left[u^{(i)}\right]_{\rho(\lambda_i)\neq 0},\left[y^{(i)}\right]_{\rho(\lambda_i)
\neq 0}$, and $\left[\hat{y}^{(i)}\right]_{\rho(\lambda_i)\neq 0}$, denote the 
matrices whose columns are formed by those vectors $u^{(i)},\ y^{(i)}$, and 
$\hat{y}^{(i)}$, for which $\rho(\lambda_i) \neq 0,\ i=1,\ldots,n$, respectively. If 
the rank of matrices 
$\sum_{j=1}^{N} \omega_jB(\zeta_j)^{-1}F(\zeta_j)S(\zeta_j)^{-1}$ and $\sum_{j=1}^{N} \omega_j S(\zeta_j)^{-1}$ is equal to that of matrices $\left[u^{(i)}\right]_{\rho(\lambda_i)\neq 0}$ and $\left[y^{(i)}\right]_{\rho(\lambda_i)\neq 0}$, respectively, then:
\begin{equation} \label{eq:nb2}
\mathtt{range}\left(\left[u^{(i)}\right]_{\rho(\lambda_i)\neq 0}\right) = 
\mathtt{range}\left(\sum_{j=1}^{N} \omega_jB(\zeta_j)^{-1}F(\zeta_j)S(\zeta_j)^{-1} \right),
\end{equation}and
\begin{equation}\label{eq:nb22}
\mathtt{range}\left(\left[y^{(i)}\right]_{\rho(\lambda_i)\neq 0}\right) =
\mathtt{range}\left(\sum_{j=1}^{N} \omega_jS(\zeta_j)^{-1}  \right).
\end{equation}

\end{theorem} 
\begin{proof}
First, notice that
 $\texttt{range}\left(\left[\rho(\lambda_i)u^{(i)}\right]_{\rho(\lambda_i)\neq 0}\right) = \texttt{range}\left(\left[u^{(i)}\right]_{\rho(\lambda_i)\neq 0}\right)$, and 
 $\texttt{range}\left(\left[\rho(\lambda_i)y^{(i)}\right]_{\rho(\lambda_i)\neq 0}\right) = \texttt{range}\left(\left[y^{(i)}\right]_{\rho(\lambda_i)\neq 0}\right).$ 
Second, we have 
$$\sum_{j=1}^{N} \omega_jB(\zeta_j)^{-1}F(\zeta_j)S(\zeta_j)^{-1} = 
  \left[\rho(\lambda_i)u^{(i)}\right]_{\rho(\lambda_i)\neq 0} \left[\hat{y}^{(i)}\right]_{\rho(\lambda_i)\neq 0}^H,$$ and 
 $$\sum_{j=1}^{N} \omega_jS(\zeta_j)^{-1} = \left[\rho(\lambda_i)y^{(i)}\right]_{\rho(\lambda_i)\neq 0}  \left[\hat{y}^{(i)}\right]_{\rho(\lambda_i)\neq 0}^H.$$ 
 Recall now that for two matrices $X_1$ and $X_2$, if the rank of the matrix $X_1X_2$ 
 is equal to that of $X_1$, then the span of the columns of $X_1$ is equal to the range 
 of $X_1X_2$. The results in (\ref{eq:nb2}) and \eqref{eq:nb22} 
 follow directly by setting  $X_1=\left[\rho(\lambda_i) u^{(i)}\right]_{\rho(\lambda_i)\neq 0}$ in \eqref{eq:nb2} 
 and $X_1=\left[\rho(\lambda_i) y^{(i)}\right]_{\rho(\lambda_i)\neq 0}$ 
 in \eqref{eq:nb22}, respectively, while 
 $X_2=\left[\hat{y}^{(i)}\right]_{\rho(\lambda_i)\neq 0}^H$.
\end{proof}

Theorem \ref{pro2} implies that a necessary condition for \eqref{eq:nb2} 
and \eqref{eq:nb22} to hold is
{\small 
\begin{equation}\label{eq:nb3}
{\rm max}\left(\texttt{rank}\left(\left[u^{(i)}\right]_{\rho(\lambda_i)\neq 0}\right),
\texttt{rank}\left(\left[y^{(i)}\right]_{\rho(\lambda_i)\neq 0}\right)\right) 
\leq \texttt{rank}\left(\left[\hat{y}^{(i)}\right]_{\rho(\lambda_i)\neq 0}\right).
\end{equation}
}
For symmetric eigenvalue problems we have $y^{(i)}=\hat{y}^{(i)}$ and 
(\ref{eq:nb3}) is trivially satisfied \cite{kalanthesis}. In practice, 
the violation of \eqref{eq:nb3} for non-Hermitian eigenvalue problems 
is quite rare in a finite-precision arithmetic environment. 

Algorithm \ref{alg:0} outlines a matrix partitioning procedure to build 
the HRR projection subspace in (\ref{hrr}) by setting the latter subspace 
equal to the direct sum of subspaces $\texttt{range} \left(\sum_{j=1}^{N} 
\omega_jB(\zeta_j)^{-1} F(\zeta_j)S(\zeta_j)^{-1}\right)$ and 
$\texttt{range} \left(\sum_{j=1}^{N} \omega_jS(\zeta_j)^{-1}\right)$. 
Each instance of Algorithm \ref{alg:2} called in Algorithm \ref{alg:0} 
performs a number of iterations which is at most equal to the number 
of eigenvalues $\lambda$ for which $\rho(\lambda) \neq 0$. Moreover,  
the two instances of Algorithm \ref{alg:2} shown in Steps 1 and 2 are 
performed in parallel and thus the linear system solutions  computed 
in Step 2 are exploited at Step 1 as well. Moreover, similarly to 
Algorithm \ref{alg:00}, Algorithm \ref{alg:0} requires no estimation 
of the value of $n_{ev}$. 


\vspace{0.1in}
\noindent\fbox{%
\begin{minipage}{\dimexpr\linewidth-9\fboxsep-9\fboxrule\relax}
\vbox{
\begin{algorithm0}{~} \label{alg:0}
\betab
0a.\> Inputs: $N,\ {\cal D}$ \\
0b.\> Compute the complex pairs $\{\omega_j,\zeta_j\}_{j=1,2,\ldots,N}$, set $G:=W:=0$    \\
0c.\> (Optionally) Reorder $(A,M)$ as in Section \ref{ddtec}    \\
1.\> Compute an orthonormal basis $G$ of $\texttt{range} \left(\sum_{j=1}^{N} \omega_jS(\zeta_j)^{-1}\right)$\\ \>by Algorithm \ref{alg:2} \\
2.\> Compute an orthonormal basis $W$ of $\texttt{range} \left(\sum_{j=1}^{N} 
\omega_jB(\zeta_j)^{-1} F(\zeta_j)S(\zeta_j)^{-1}\right)$\\ \>by Algorithm \ref{alg:2} \\
3.\> Set $Z = \begin{bsmallmatrix} W & \\ & G \end{bsmallmatrix}$, 
solve the eigenvalue problem in (\ref{hrr}) and return \\\> all
Ritz values $\theta \in {\cal D}$ and associated Ritz vectors
\entab
\end{algorithm0}
}\vspace{0.05in}
\end{minipage}
}
\vspace{0.1in}

{
Unless mentioned otherwise, the default value of the number of poles in the rational 
filter $\rho$ will be equal to $N=16$.}



\subsection{Second algorithm} \label{second_alg}

This section describes an alternative technique to construct the matrix $W$ 
in Algorithm \ref{alg:0} under the assumption that the pencil $(B,M_B)$ 
is diagonalizable. Throughout the rest of this section we will denote 
the eigentriplets of the pencil $(B,M_B)$ by $\left(\delta_i,v^{(i)},
\hat{v}^{(i)}\right),\ i=1,2,\ldots,d$, where $\delta_i$ denotes the 
eigenvalue of $(B,M_B)$ with the $i$th shortest distance from the center 
of the disk ${\cal D}$, and $v^{(i)}$ and $\left(\hat{v}^{(i)}\right)^H$ 
denote the corresponding right and left eigenvectors, respectively.

By combining (\ref{ideal}) and (\ref{eq1}), we can write the top 
$d\times 1$ part of the eigenvector 
$x^{(i)}=\left(\begin{smallmatrix}u^{(i)} \\ y^{(i)}\end{smallmatrix}\right)$ 
associated with the eigenvalue $\lambda_i$ as
\begin{equation}\label{range:eq2}
\begin{aligned}
u^{(i)} & = -B(\lambda_i)^{-1}F(\lambda_i)y^{(i)}.
\end{aligned}
\end{equation} 
While expression (\ref{range:eq2}) is not practical, it serves as a 
starting point for the construction of a subspace which (approximately) 
captures $\texttt{span}\left(u^{(i)}\right)$ without depending on 
the (unknown) quantities $\lambda_i$ and $y^{(i)}$. 

Let $G$ be a matrix such that $y^{(i)} \in \texttt{range}(G)$, e.g., the matrix 
$G$ constructed in Algorithm \ref{alg:0}. In addition, define the matrices
$$V_{\phi} = \left[v^{(1)},v^{(2)},\ldots,v^{(\phi)}\right]\quad \text{and}
\quad \hat{V}_{\phi} = \left[\hat{v}^{(1)},\hat{v}^{(2)},\ldots,\hat{v}^{(\phi)}
\right],$$ where $\phi \in {\cal Z}^*$ is larger than or equal to the number 
of eigenvalues of $(B,M_B)$ located inside the disk ${\cal D}$. Taking advantage 
of the identity $I=V_{\phi}\hat{V}_{\phi}^HM_B + (I-V_{\phi}\hat{V}_{\phi}^HM_B)$,
and noticing that 
$\mathtt{span}\left(V_{\phi}\hat{V}_{\phi}^HM_BB(\lambda_i)^{-1}
F(\lambda_i)y^{(i)}\right)
\subseteq \mathtt{span}\left(v^{(1)},v^{(2)},\ldots,v^{(\phi)}\right)$,
we can write 
{\small \begin{equation}\label{range:eq3a}
\begin{aligned}
\mathtt{span}\left(u^{(i)}\right) 
=& ~ \mathtt{span} \left(B(\lambda_i)^{-1}F(\lambda_i)y^{(i)}\right)\\
=& ~ \mathtt{span} \left(V_{\phi}\hat{V}_{\phi}^HM_BB(\lambda_i)^{-1}F(\lambda_i)y^{(i)}
+(I-V_{\phi}\hat{V}_{\phi}^HM_B)B(\lambda_i)^{-1}F(\lambda_i)y^{(i)}\right)\\
\subseteq & ~\mathtt{span}\left(v^{(1)},v^{(2)},\ldots,v^{(\phi)}\right)+ \mathtt{span}\left((I-V_{\phi}\hat{V}_{\phi}^HM_B)B(\lambda_i)^{-1} F(\zeta)G\right)+\\
& ~\mathtt{span}\left((I-V_{\phi}\hat{V}_{\phi}^HM_B)B(\lambda_i)^{-1} M_FG\right),
\end{aligned}
\end{equation}}
where $\zeta \in {\cal D}$, and we replaced $F(\lambda_i)$ by its 
equivalent form $$F(\lambda_i)=F(\zeta)-(\lambda_i-\zeta)M_F.$$

The expression in (\ref{range:eq3a}) still depends on $\lambda_i$ through the 
term $B(\lambda_i)^{-1}$. Next, we show an equivalent expression of the matrix  $(I-V_{\phi}\hat{V}_{\phi}^HM_B)B(\lambda_i)^{-1}$.

\begin{theorem} \label{lem1}
Let $\zeta_c \in \mathbb{C}$ be the center of disk ${\cal D}$ and $\phi \in 
\mathbb{N}$ larger than or equal to the number of eigenvalues of $(B,M_B)$ 
located inside ${\cal D}$. If we define the matrix
\begin{equation*}
    \widetilde{B}(\zeta) := \left(I-V_\phi \hat{V}_\phi^HM_B\right)B(\zeta)^{-1},
\end{equation*}
then  
\begin{equation}\label{range:eq4}
\left(I-V_\phi \hat{V}_\phi^HM_B\right)B(\lambda_i)^{-1} =
\widetilde{B}(\zeta_c)
\sum_{k=0}^{\infty} \left[(\lambda_i-\zeta_c)M_B\widetilde{B}(\zeta_c)\right]^k,
\end{equation}
for any $\lambda_i \in {\cal D}$.
\end{theorem}

\begin{proof}
Define the matrices  
$$V = \left[V_{\phi},v^{(\phi+1)},\ldots,v^{(d)}\right]\quad \text{and}\quad \hat{V} = \left[\hat{V}_{\phi},\hat{v}^{(\phi+1)},\ldots,\hat{v}^{(d)}\right].$$
Recall that $\hat{V}^HM_BV=I$, and thus $M_B =\hat{V}^{-H}V^{-1}$
and $B =\hat{V}^{-H}\left( \begin{smallmatrix}\delta_1 &  & \\ 
 & \ddots & \\ & & \delta_d\end{smallmatrix}\right)V^{-1}$. 
Using the above identities we can write 
\begin{equation*}
\begin{aligned}
\widetilde{B}(\zeta) &=\left(I-V_\phi \hat{V}_\phi^HM_B\right) V
\left(\begin{smallmatrix} \delta_1-\zeta & &\\ & \ddots &\\ & & \delta_d-\zeta \end{smallmatrix}\right)^{-1}\hat{V}^H\\
 &= 
  V
  \left( 
  \begin{smallmatrix}
  \scalebox{2}{$0$}_{\phi,\phi}  & & & & \\ 
  &   & \dfrac{1}{\delta_{\phi+1}-\zeta} & & \\ 
  &   & & \ddots & \\ 
  &  & & & \dfrac{1}{\delta_{d}-\zeta}\\
 \end{smallmatrix}
 \right)\hat{V}^{H}.
 \end{aligned}
 \end{equation*} 
Let us now define the scalar $\gamma_j =\dfrac{\lambda_i-\zeta_c}{\delta_j-\zeta_c}$. We can write 
\begin{equation*}
    \widetilde{B}(\zeta_c)\left[(\lambda_i-\zeta_c)M_B\widetilde{B}(\zeta_c)\right]^k = V
    \begin{pmatrix}
    \scalebox{2}{$0$}_{\phi,\phi} & & & \\[0.3em]
    & \dfrac{\gamma_{\phi+1}^k}{\delta_{\phi+1}-\zeta_c} & & \\[0.3em]
    & & \ddots & \\[0.3em]
    & & & \dfrac{\gamma_{d}^k}{\delta_d-\zeta_c} \\[0.3em]
    \end{pmatrix}
    \hat{V}^{H}.
\end{equation*}
Accounting for all powers $k=0,1,2,\ldots$, gives
{\small \begin{equation*}
    \widetilde{B}(\zeta_c)\sum_{k=0}^{\infty} 
    \left[(\lambda_i-\zeta_c)M_B\widetilde{B}(\zeta_c)\right]^k = V
    \begin{pmatrix}
    \scalebox{2}{$0$}_{\phi,\phi} & & & \\[0.3em]
    & \dfrac{\sum_{k=0}^{\infty} \gamma_{\phi+1}^k}{\delta_{\phi+1}-\zeta_c} & & \\[0.3em]
    & & \ddots & \\[0.3em]
    & & & \dfrac{\sum_{k=0}^{\infty} \gamma_{d}^k}{\delta_{d}-\zeta_c} \\[0.3em]
    \end{pmatrix}
    \hat{V}^{H}.
\end{equation*}}
Since $\zeta_c$ is the center of ${\cal D}$, it follows that 
$|\gamma_j| < 1$ for any $\delta_j \notin {\cal D}$. 
Therefore, the geometric series converges and $\sum_{k=0}^{\infty} \gamma_j^k=\dfrac{1}{1-\gamma_j}=\dfrac{\delta_j-\zeta_c}{\delta_j-\lambda_i}$. 
It follows that $\dfrac{1}{\delta_j-\zeta_c}\sum_{k=0}^{\infty} \gamma_j^k
=\dfrac{1}{\delta_j-\lambda_i}$.

We finally have
\begin{equation*}
\begin{aligned}
    \widetilde{B}(\zeta_c)\sum_{k=0}^{\infty} 
    \left[(\lambda_i-\zeta_c)M_B\widetilde{B}(\zeta_c)\right]^k & = V
    \begin{pmatrix}
    \scalebox{2}{$0$}_{\phi,\phi} & & & & \\[0.3em]
    & \dfrac{1}{\delta_{\phi+1}-\lambda_i} & & \\[0.3em]
    & & \ddots & \\[0.3em]
    & & & \dfrac{1}{\delta_{d}-\lambda_i} \\[0.3em]
    \end{pmatrix}
    \hat{V}^{H} \\
    & = \left(I-V_\phi \hat{V}_\phi^HM_B\right)B(\lambda_i)^{-1}.
    \end{aligned}
\end{equation*}
This concludes the proof.
\end{proof}

Theorem \ref{lem1} implies that we can approximate 
$\left(I-V_\phi \hat{V}_\phi^HM_B\right)B(\lambda_i)^{-1}$ 
through a finite truncation of the right-hand side 
in (\ref{range:eq4}). The approximation error of this truncation is
considered in the following proposition.
\begin{proposition}\label{prop2}
Let $\psi$ be a positive integer and define the error matrix 
\begin{equation*}
R_{\psi}(\lambda_i)=
(I-V_{\phi}\hat{V}_{\phi}^HM_B)B(\lambda_i)^{-1}-\widetilde{B}(\zeta_c)
\sum_{k=0}^\psi \left[(\lambda_i-\zeta_c)M_B\widetilde{B}(\zeta_c)\right]^k.
\end{equation*}
Then
\begin{equation}
R_{\psi}(\lambda_i)
=
\sum_{k=\psi+1}^{\infty} \sum_{j=\phi+1}^{d}\left[\dfrac{(\lambda_i-\zeta_c)^{k}}
{(\delta_j-\zeta_c)^{k+1}}\right]v^{(j)} \left(\hat{v}^{(j)}\right)^HM_B.
\end{equation}
\end{proposition}

\begin{proof}
Recall the scalar $\gamma_j =\dfrac{\lambda_i-\zeta_c}{\delta_j-\zeta_c}$.
The matrix $R_{\psi}(\lambda_i)$ is then equal to
\begin{equation*}
    R_{\psi}(\lambda_i) = V
    \begin{pmatrix}
    \scalebox{2}{$0$}_{\phi,\phi} & & & \\[0.3em]
    & \dfrac{\sum_{k=\psi+1}^{\infty} \gamma_{\phi+1}^k}{\delta_{\phi+1}-\zeta_c} & & \\[0.3em]
    & & \ddots & \\[0.3em]
    & & & \dfrac{\sum_{\psi+1}^{\infty} \gamma_{d}^k}{\delta_{d}-\zeta_c} \\[0.3em]
    \end{pmatrix}
    \hat{V}^{H}.
\end{equation*}
The proof concludes by replacing $\gamma_j$ with its ratio.
\end{proof}
Proposition \ref{prop2} indicates that when $\zeta_c$ is close to the sought 
eigenvalues $\lambda_1,\ldots,\lambda_{nev}$ and the eigenvalues 
$\delta_{\phi+1},\ldots,\delta_{d}$ are located far away from the disk $\cal{D}$, 
then the matrix $\widetilde{B}(\zeta_c)\sum_{k=0}^\psi
\left[(\lambda_i-\zeta_c)M_B\widetilde{B}(\zeta_c)\right]^k$ can be used as 
an accurate approximation of the matrix 
$(I-V_{\phi}\hat{V}_{\phi}^HM_B)B(\lambda_i)^{-1}$ 
even for small values of $\psi$.

Based on the above discussion, we expect to find a reasonable approximation 
of the subspace $\texttt{span}\left(u^{(1)},\ldots, u^{(n_{ev})}\right)$ in  
the range of the matrix
\begin{equation}\label{coWenh}
W_{\phi,\psi} = \Big[V_\phi,\ \left(I-V_\phi \hat{V}_\phi^HM_B\right)\hat{B}_{\psi}(\zeta_c)\hat{G}_F,\ \underbrace{\left(I-V_\phi \hat{V}_\phi^HM_B\right)\hat{B}_{\psi}(\zeta_c)\hat{G}_{M_F}}_{\text{{\rm only if $M_F\neq 0$}}}\Big],
\end{equation}
where 
\begin{equation*}
\hat{B}_{\psi}(\zeta_c) = \left[\tilde{B}(\zeta_c),
\tilde{B}(\zeta_c)\left[M_B\tilde{B}(\zeta_c)\right],\ldots,
\tilde{B}(\zeta_c)\left[M_B\tilde{B}(\zeta_c)\right]^\psi\right],
\end{equation*}
and we set
\begin{equation*}
{\small \hat{G}_F=
\begin{bmatrix}
F(\zeta_c)G & & & \\[0.3em]
 & F(\zeta_c)G & & \\[0.3em]
 &  & \ddots & \\[0.3em]
 &  &  &  & F(\zeta_c)G \\[0.3em]
\end{bmatrix},\ 
\hat{G}_{M_F}=
\begin{bmatrix}
M_F G & & & \\[0.3em]
 & M_F G & & \\[0.3em]
 &  & \ddots & \\[0.3em]
 &  &  &  & M_F G \\[0.3em]
\end{bmatrix}}.
\end{equation*}

\vspace{0.1in}
\noindent\fbox{%
\begin{minipage}{\dimexpr\linewidth-5\fboxsep-5\fboxrule\relax}
\vbox{
\begin{algorithm0}{~} \label{alg:1}
\betab
\>0a.\> Inputs: $N,\ {\cal D}$, $\psi\ \mathrm{(optionally)}, \phi$ (optionally) \\
\>0b.\> Compute the complex pairs $\{\omega_j,\zeta_j\}_{j=1,2,\ldots,N}$, set $G:=0$    \\
\>0c.\> (Optionally) Reorder $(A,M)$ as in Section \ref{ddtec}    \\
\> 1.\> Compute an orthonormal basis $G$ of $\texttt{range} \left(\sum_{j=1}^{N} \omega_jS(\zeta_j)^{-1}\right)$\\ \> \>by Algorithm \ref{alg:2} \\
\>2.\> Compute the eigenpairs associated with the $\phi$ eigenvalues of smallest\\ 
\> \> modulus of the pencil $(B(\zeta),M_B)$ and form the matrix $V_\phi$\\
\>3.\> Set the matrix $W_{\phi,\psi}$ as in (\ref{coWenh})\\
\>4.\> Set $Z = \begin{bsmallmatrix} W_{\phi,\psi} & \\ & G \end{bsmallmatrix}$, 
solve the eigenvalue problem in (\ref{hrr}) and return \\\>\> all
Ritz values $\theta \in {\cal D}$ and associated Ritz vectors 
\entab
\end{algorithm0}
}\vspace{0.05in}
\end{minipage}
}
\vspace{0.1in}

{The complete algorithmic procedure is summarized 
in Algorithm \ref{alg:1}. The accuracy in the approximation of the 
eigenpairs $(\lambda_i,x^{(i)}),\ i=1,\ldots,n_{ev}$, depends on the 
distance of the eigenvalues $\lambda_i \in {\cal D}$ from both the center 
of the disk ${\cal D}$ and the (non-deflated) eigenvalues of the matrix 
pencil $(B,M_B)$. In contrast, the accuracy provided by 
Algorithm \ref{alg:0} is irrespective to the location of the eigenvalues 
$\lambda_i \in {\cal D}$. Thus, the latter should be the algorithm of 
choice when one seeks higher accuracy in the approximation of the $n_{ev}$ sought 
eigenpairs of the pencil $(A,M)$. On the other hand, Algorithm \ref{alg:1} 
should be preferred when a few digits of accuracy are deemed enough, and 
lower wall-clock execution time is critical.}

{
Compared to Algorithm \ref{alg:0}, Algorithm \ref{alg:1} introduces two 
new parameters, $\psi \in \mathbb{N}$ and $\phi \in \mathbb{N}$. Larger 
values of these two integers lead to higher accuracy but increase the 
associated computational cost. Increasing the value of $\psi$ aims at 
reducing  the error along all eigenvector directions of $(B,M_B)$, while 
increasing the value of $\phi$ aims at eliminating the approximation error 
associated with eigenvectors corresponding closer to the center of the 
disk ${\cal D}$. Generally speaking, the main improvements in accuracy 
come from increasing the value of $\psi$. Our default choice is to  
set $\psi=1$, and $\phi$ equal to the number of eigenvalues of the pencil 
$(B,M_B)$ located inside the disk ${\cal D}$. If additional accuracy is 
needed, one can augment $W_{\phi,\psi}$ with either additional eigenvectors 
of the pencil $(B,M_B)$ (i.e., increase $\phi$), or additional resolvent 
approximation matrix terms (i.e., increase $\psi$) and only repeat the 
Rayleigh-Ritz projection step. This approach can be repeated more than 
once, i.e., until the residual norms of all $n_{ev}$ approximate eigenpairs 
are less a chosen threshold.}

\section{Practical details} \label{sec4a}


\subsection{Computational cost comparison} \label{prade}

The main computational bottleneck of the rational filtering algorithms discussed 
in this paper is the solution of complex-shifted sparse linear systems of the form 
$B(\zeta)x_d=b_d$ and $S(\zeta)x_s=b_s$. Therefore, an algorithm that requires 
fewer such linear system solutions will typically be faster as well.

\begin{table}
\caption{Total number of linear system solutions of the form $B(\zeta)x_d=b_d$ 
and $S(\zeta)x_s=b_s$ computed by Algorithm \ref{alg:00}, Algorithm \ref{alg:0}, 
Algorithm \ref{alg:1}, and the algorithm used in the FEAST software package. 
The variables $\eta_1\in\mathbb{N},\ \eta_2\in\mathbb{N},\ \eta_3\in\mathbb{N}$, 
denote the number of iterations performed by Algorithm \ref{alg:2} when called 
from Algorithm \ref{alg:00}, Algorithm \ref{alg:0}, and Algorithm \ref{alg:1}, respectively. The variable $\tau_{\phi}$ denotes the number of linear systems 
of the form $B(\zeta_c)x_d=b_d$ required to compute the $\phi$ sought 
eigenvectors of the pencil $(B-\zeta_c M_B,M_B)$ by Implicitly Restarted 
Arnoldi (IRA) combined with shift-and-invert \cite{lehoucq1998arpack}. 
The variable $\eta_4\in\mathbb{N}$ denotes the number of iterations 
performed by subspace iteration. } \label{tab:table22}
\centering
{\small \begin{tabular}{lccccl}
\toprule
& \multicolumn{2}{c}{} & \multicolumn{2}{c}{Alg. \ref{alg:1}} & \multicolumn{1}{c}{}
\\\cmidrule(lr){4-5}
 & Alg. \ref{alg:00} & Alg. \ref{alg:0} & $M_F=0$ & $M_F \neq 0$ & Sub. It.\\\midrule
$B(\zeta)x_d=b_d$    & $2N\eta_1$ & $N\eta_2$ & $\eta_3(\psi+1)+\tau_{\phi}$ & $2\eta_3(\psi+1)+\tau_{\phi}$ & $2mN\eta_4$  \\
$S(\zeta)x_s=b_s$    & $N\eta_1$ & $N\eta_2$ & $N\eta_3$ & $N\eta_3$ & $mN\eta_4$  \\
\bottomrule
\end{tabular}}
\end{table}

Table \ref{tab:table22} summarizes the computational costs of Algorithm 
\ref{alg:00}, Algorithm \ref{alg:0}, and Algorithm  \ref{alg:1}, where 
we assume that all linear systems are solved by a direct solver and their 
complexity is oblivious to the actual value of $\zeta \notin \Lambda(B,M_B)$. 
The variables $\eta_1\in\mathbb{N},\ \eta_2\in\mathbb{N},\ \eta_3\in\mathbb{N}$, 
denote the number of iterations performed by Algorithm \ref{alg:2} when called 
from Algorithm \ref{alg:00}, Algorithm \ref{alg:0}, and Algorithm \ref{alg:1}, respectively. It is straightforward to observe that 
when $\eta_1\approx \eta_2\approx \eta_3$, Algorithm \ref{alg:0} requires 
about half linear system solutions of the form $B(\zeta)x_d=b_d$ than what 
Algorithm \ref{alg:00} does. Moreover, Algorithm \ref{alg:1} requires a 
number of linear system solutions which is independent of the number 
of poles $N$. Thus, larger values of $N$ should increase the computational 
complexity gap in favor of Algorithm \ref{alg:1}. For comparison purposes we 
also list the computational complexity of subspace iteration applied to matrix $\rho(M^{-1}A)$ with an initial subspace of dimension $m\geq n_{ev}$.
In contrast to the algorithms proposed in this paper, the convergence of 
subspace iteration depends on the dimension $m$ of its initial subspace.

\subsection{Matrix partitionings} \label {ddtec}

The matrix partitioning algorithms discussed in this paper can take 
advantage of a reordering of the pencil $(A,M)$ so that the pencil 
$(B,M_B)$ is block-diagonal. For Algorithm \ref{alg:1} this implies 
that the computation of the matrix $V_\phi$ then decouples into $p$ 
independent generalized non-Hermitian eigenvalue problems. 
The eigenvalues of each one of these $p$ matrix pencils can be then 
computed in parallel.

To obtain the above reordering we partition the adjacency graph of the 
matrix 
$\left|A\right| + \left|A^T\right|+\left|M\right| + \left|M^T\right|$ 
into $p \geq 2$ non-overlapping partitions \cite{saad1996ilum}. We then 
reorder the equations/unknowns so that the interior variables across all 
partitions are ordered before the interface ones. The latter procedure 
is equivalent to transforming the original pencil $(A,M)$ into the form $\left(PAP^T,PMP^T\right)$, where the $n\times n$ matrix $P$ holds the 
row permutation of $(A,M)$. The eigenpairs of $(A,M)$ are connected with 
those of the matrix pencil $\left(PAP^T,PMP^T\right)$ through the formula 
\begin{equation*}PAP^T\left(Px^{(i)}\right)=\lambda_iPMP^T\left(Px^{(i)}
\right).\end{equation*}


The matrices $PAP^T$ and $PMP^T$ can be written us
{\small \begin{equation*} \label{dd1}
 PAP^T = 
 \xpmatrix{
 B_1      &        &           &          &  F_{1 }  \cr
          & B_2    &           &          &  F_{2 }  \cr
          &        &   \ddots  &          &  \vdots  \cr
          &        &           &  B_p     &   F_p    \cr
  E_1   & E_2  &   \ldots  &  E_p   &   C      \cr
 },\ \ 
 PMP^T = 
 \xpmatrix{
 M_B^{(1)}   &           &           &             &  M_F^{(1)}  \cr
             & M_B^{(2)} &           &             &  M_F^{(2)}  \cr
             &           &   \ddots  &             &  \vdots     \cr
             &           &           &  M_B^{(p)}  &  M_F^{(p)}  \cr
  M_E^{(1)} & M_E^{(2)} &  \ldots  & M_E^{(p)} &   M_C       \cr
 }
 \end{equation*}}
 where matrices $B_i$ and $M_B^{(i)}$ are square matrices of size 
 $d_i\times d_i$, matrices $F_i$  $\left(E_i\right)$ and $M_F^{(i)}$ $\left(M_E^{(i)}\right)$ are of size $d_i\times s_i$  ($s_i\times 
 d_i$), and the integers $d_i$ and $s_i$ denote the number of interior 
 and interface  nodes located in the $i$th subdomain of the adjacency 
 graph of $\left|A\right| + \left|A^T\right|+\left|M\right| + 
 \left|M^T\right|$, respectively. On the other hand, matrices $C$ and 
 $M_C$ are of size $s\times s$ where $s=\sum\limits_{i=1}^p s_i$.

\section{Experiments} \label{sec4}

The numerical experiments presented in this section were performed in a Matlab 
environment (version R2018b), using 64-bit arithmetic, on a single core of a 
MacBook Pro equipped with a quad-core 2.5 GHz Intel Core i7 processor and 16 GB 
1600 MHz DDR3 of system memory. The matrices used throughout our experiments 
are listed in Table \ref{testmat} and can be retrieved from SuiteSparse Matrix 
Collection \cite{UFSM} and Matrix Market repository \cite{boisvert1997matrix}. 

\begin{table}[tbhp]
\centering
{\footnotesize
\caption{$n$: size of matrices $A$ and $M$, $nnz(.)$: number of nonzero 
entries.  \label{testmat}}
\begin{tabular}{lcccccc}
\hline\hline
\# &Matrix pencil        & $n$   & $nnz(A)/n$ & $nnz(M)/n$ & Application\\
\hline\hline
1.& \texttt{bfw782}     & 782      & 9.6 & 7.6    &  Engineering\\
2.& \texttt{utm1700b}   & 1,700    & 12.7 & 1.0   & Electromagnetics \\
3.& \texttt{wang1}      & 2,903    & 6.6 & 1.0    & Semiconductors \\
4.& \texttt{rdb3200l}   & 3,200    & 5.9 & 1.0    & CFD \\
5.& \texttt{thermal}    & 3,456    & 19.2 & 1.0   & Thermal \\
6.& \texttt{dw4096}     & 8,192    & 5.1 & 1.0    & Engineering \\
7.& \texttt{big}        & 13,209   & 6.9 & 1.0    & Directed weighted graph \\
\hline\hline
\end{tabular}}
\end{table}

Throughout the rest of this section we consider the application of 
three different algorithms: a) Algorithm \ref{alg:0}, b) Algorithm 
\ref{alg:1}, and finally c) subspace iteration with the matrix 
$\rho(M^{-1}A)$, where the initial subspace is of dimension $m\geq 
n_{ev}$. We will refer to this approach as RSI. As a separate note, 
a high-performance implementation of subspace iteration with rational 
filtering can be found in the FEAST software package.

The radius of the disk ${\cal D}$ is set equal to 1.001 times the radius 
of the minimal enclosing circle of eigenvalues $\lambda_1,\ldots,
\lambda_{n_{ev}}$. The rational filter function in (\ref{ratfilter}) is 
constructed through discretizing (\ref{eq:contour}) by the trapezoidal 
rule of order $N$. Throughout the rest of this section we assume that 
the iterative loop in Algorithm \ref{alg:2} terminates when the ratio of 
the smallest to the largest singular value is less than or equal to $1.0 
\times 10^{-12}$, and we set a maximum number of four hundred iterations. 
All matrix pencils were reordered as discussed in Section \ref{ddtec} 
using $p=8$. The residual norm of each approximate eigenpair $(\hat{\lambda},
\hat{x})$ is computed as
$
    \hat{\rho} = \dfrac{\|A\hat{x}-\hat{\lambda}M\hat{x}\|_2}
    {\|A\hat{x}\|_2+|\hat{\lambda}|\|M\hat{x}\|_2}.
$
All algorithms discussed in this section return only those approximate 
eigenpairs $(\hat{\lambda},\hat{x})$ for which $\hat{\lambda} \in 
{\cal D}$. When more than $n_{ev}$ approximate eigenvalues are located 
in ${\cal D}$ we purge the spurious ones by keeping only those for which 
the associated residual norm is smaller than the threshold tolerance 
$1.0 \times 10^{-3}$. This approach was successful in all experiments we 
performed.

\subsection{A detailed example}

\begin{figure} 
\includegraphics[width=0.49\textwidth]{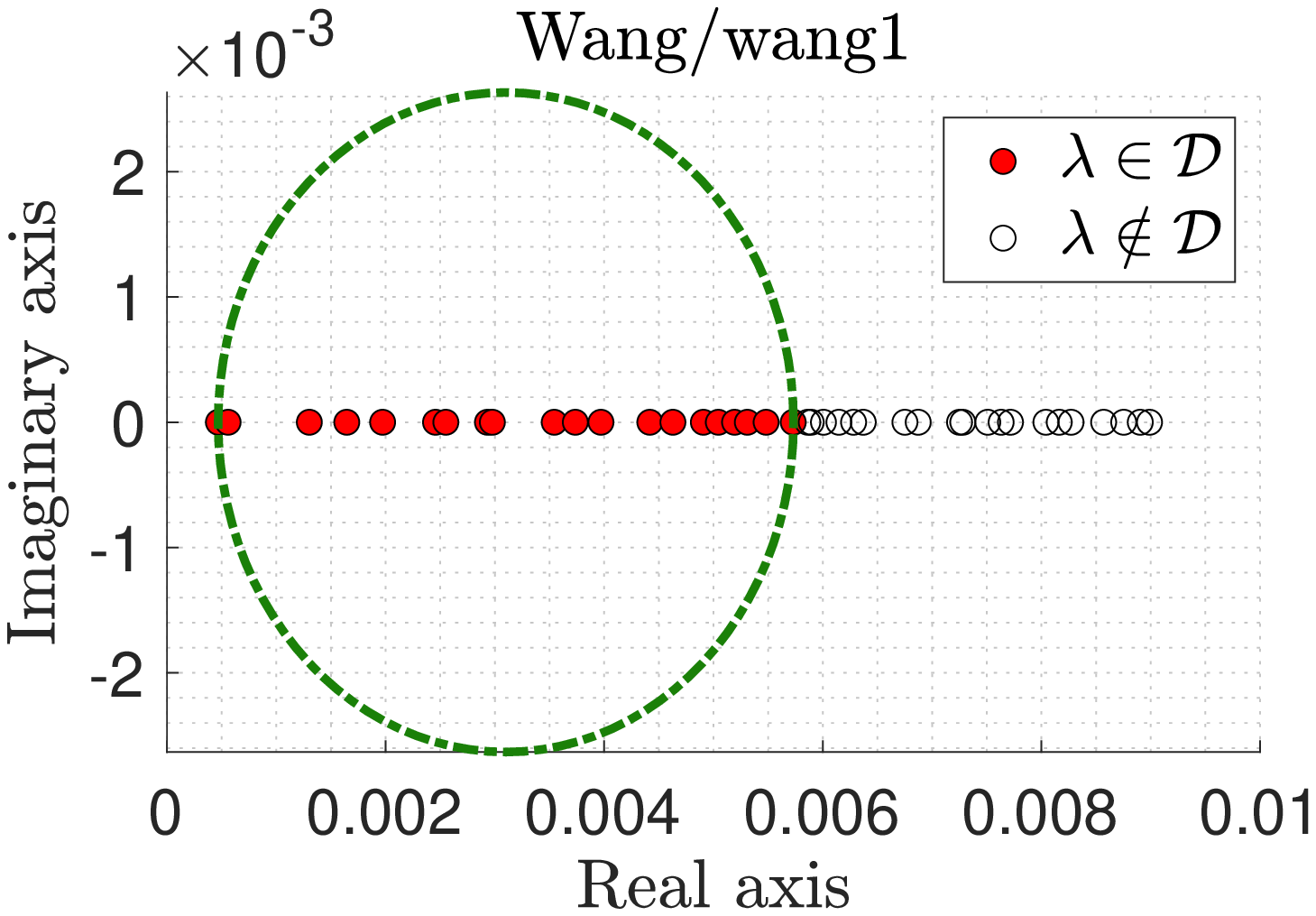}
\includegraphics[width=0.49\textwidth]{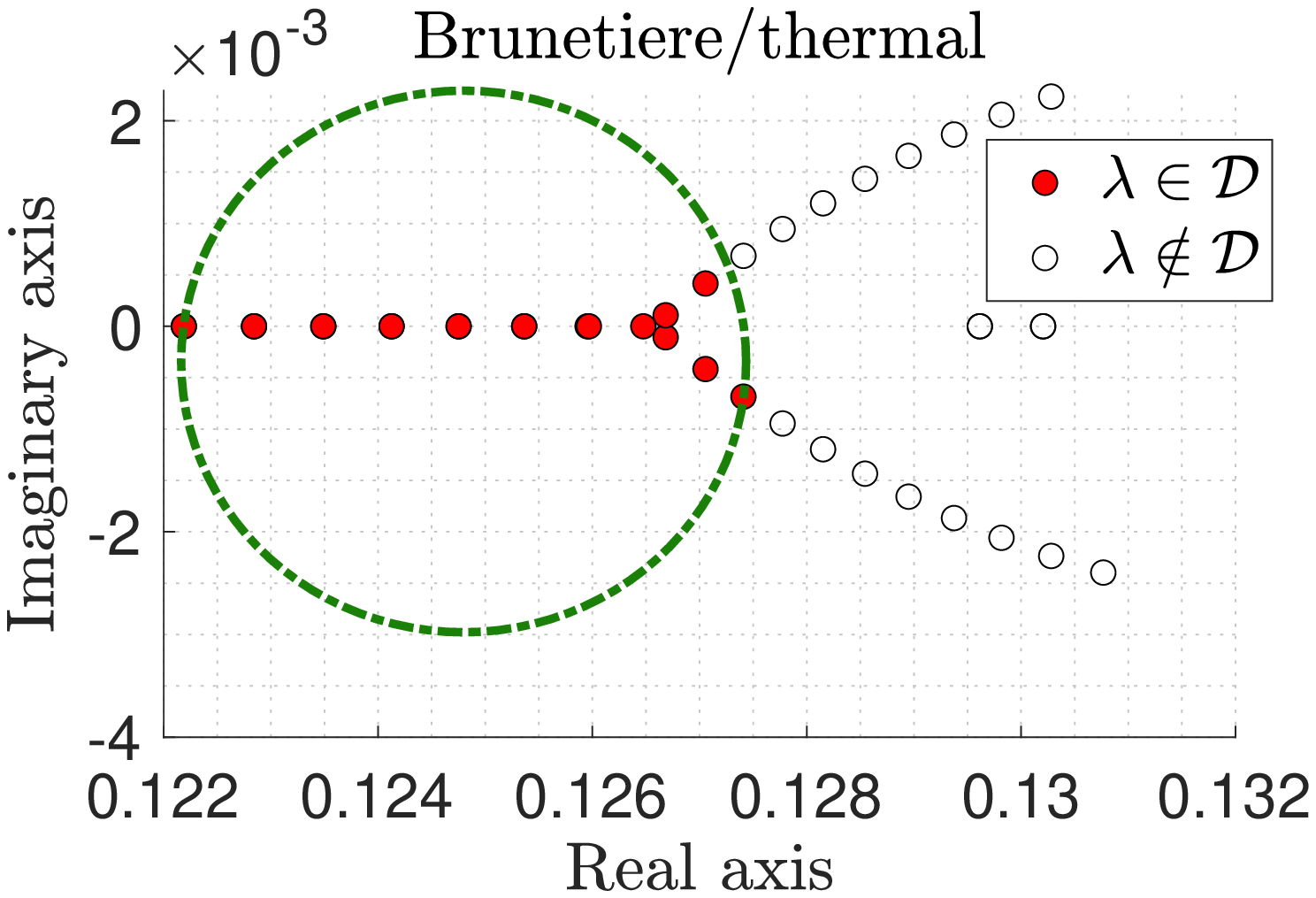}
\includegraphics[width=0.49\textwidth]{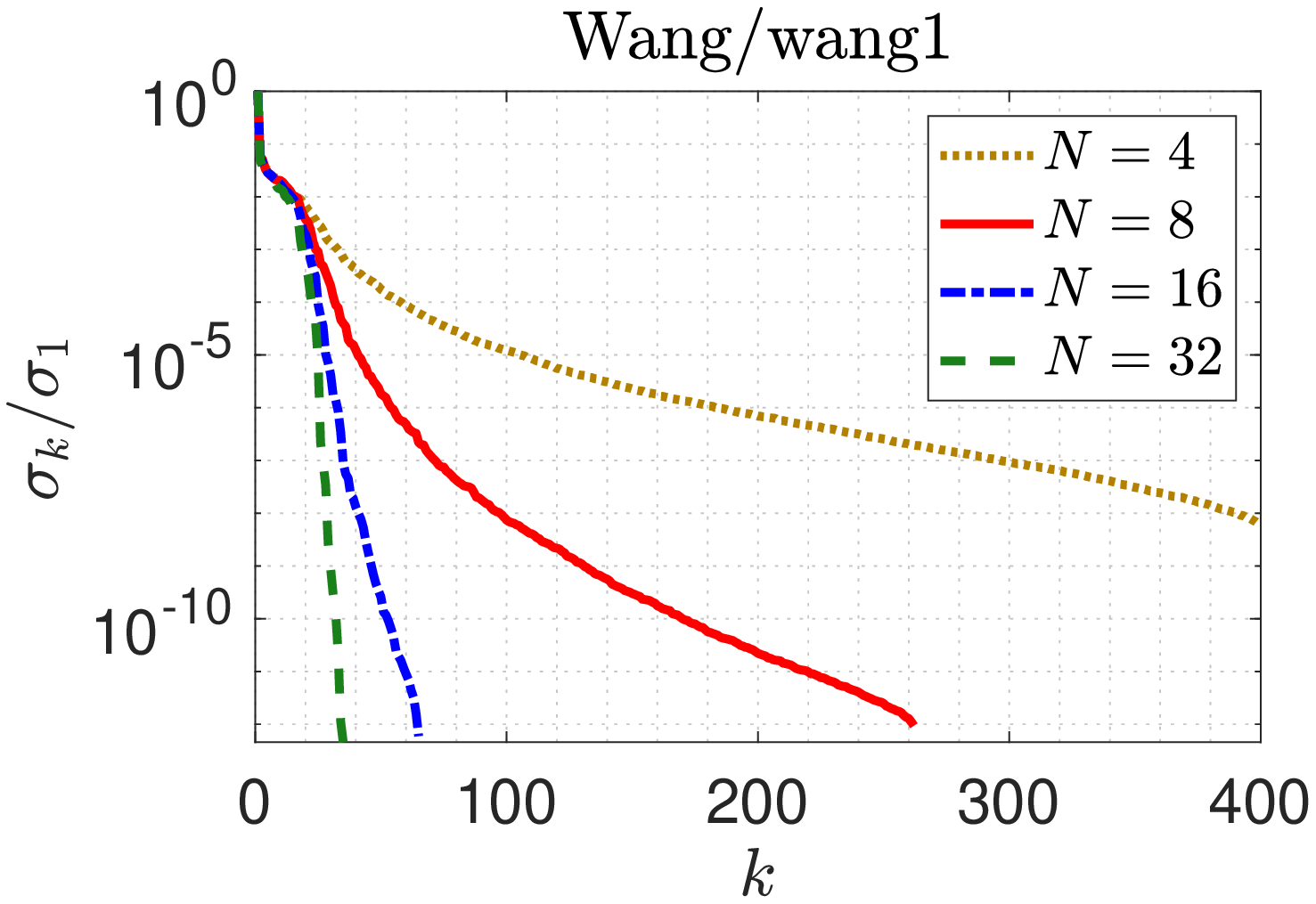}
\includegraphics[width=0.49\textwidth]{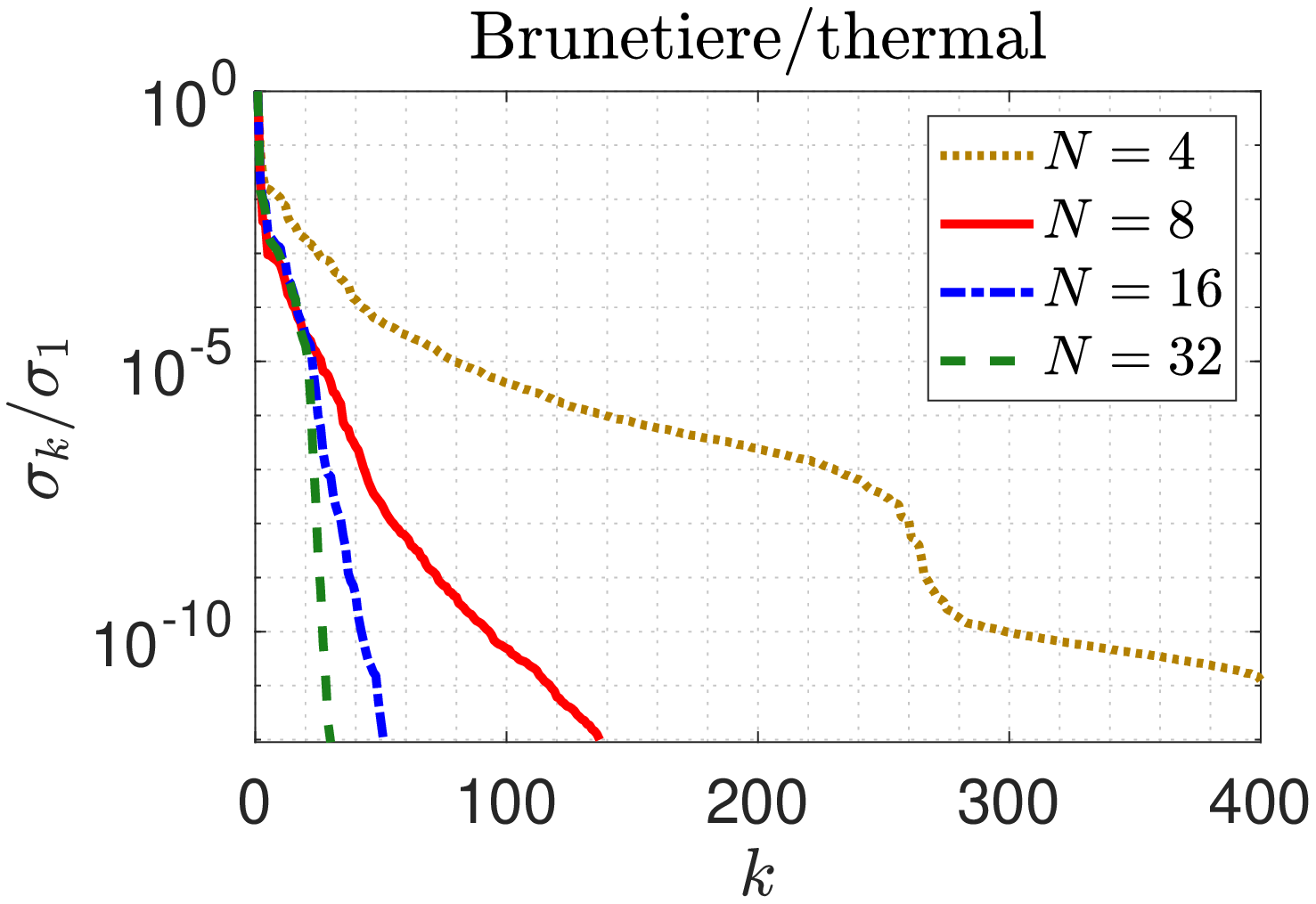}
\includegraphics[width=0.49\textwidth]{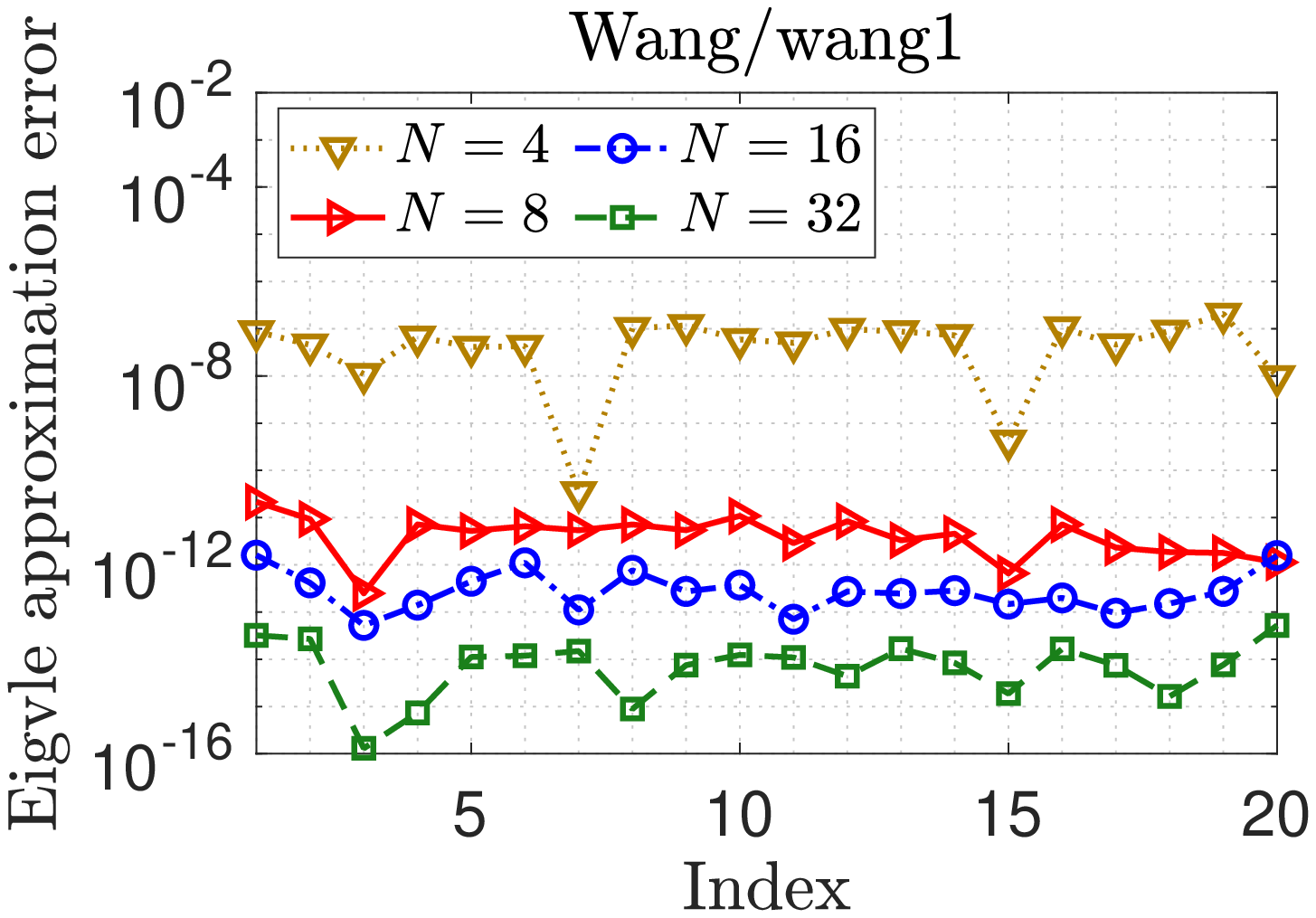}
\includegraphics[width=0.49\textwidth]{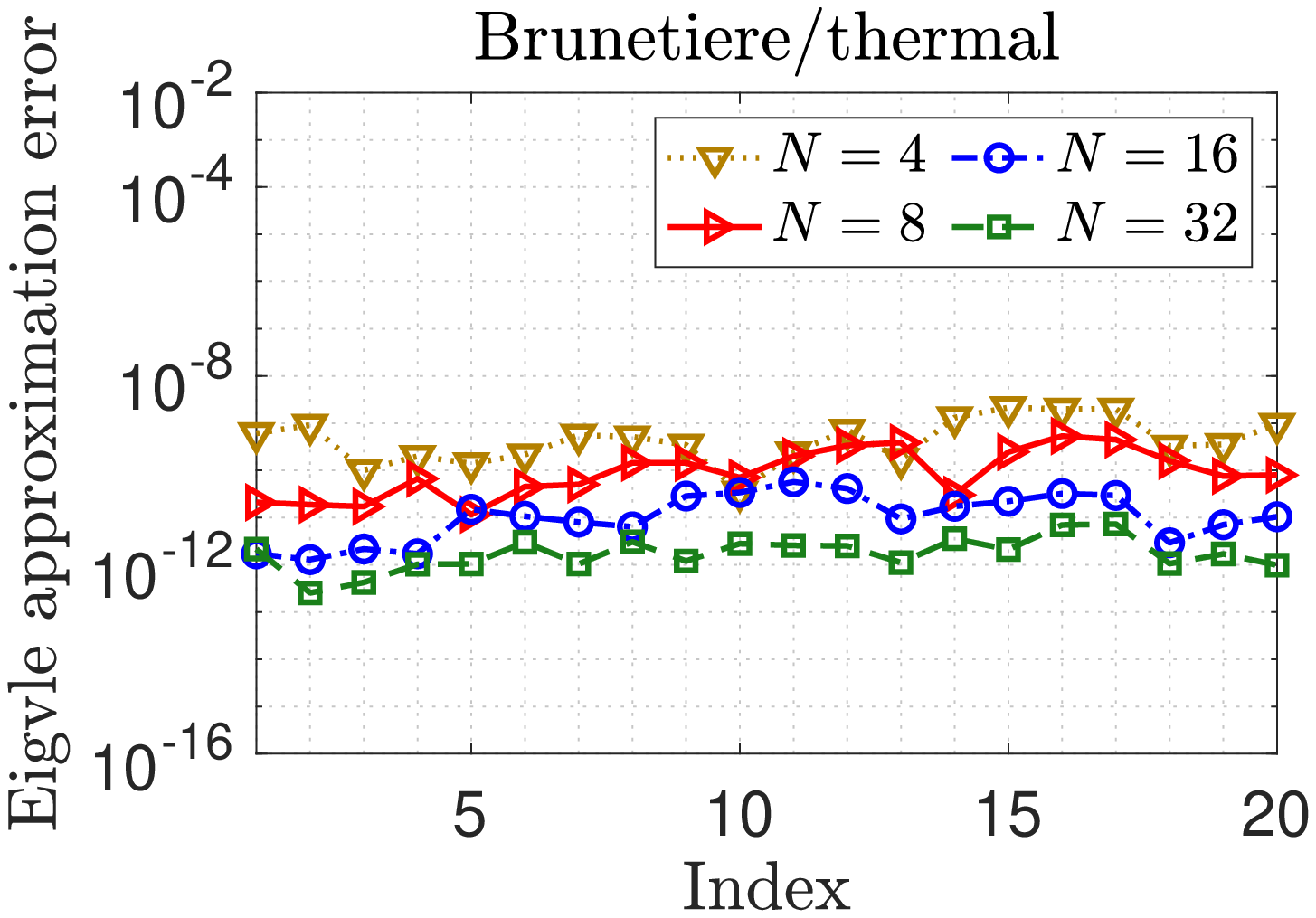}
\includegraphics[width=0.49\textwidth]{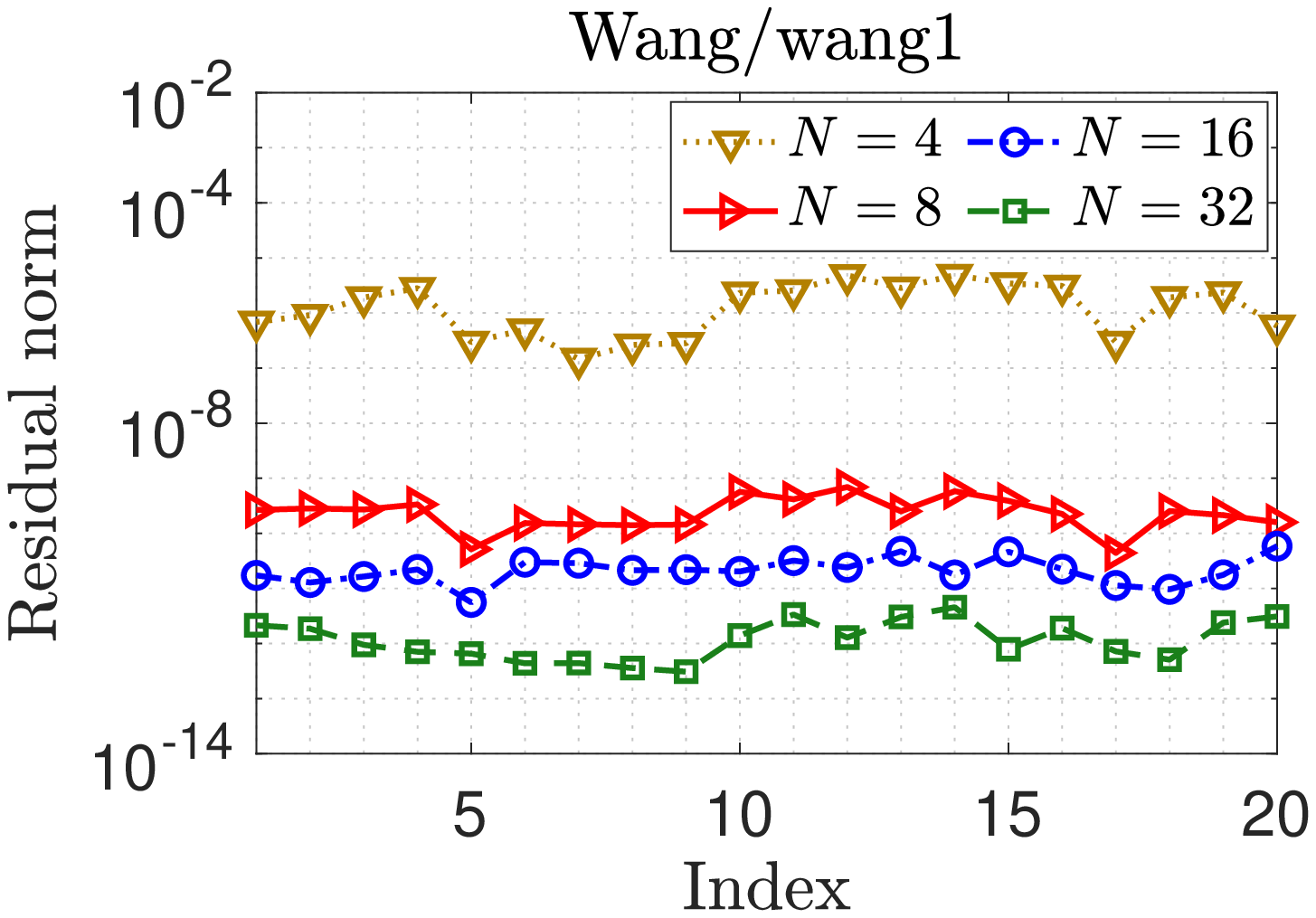}
\includegraphics[width=0.49\textwidth]{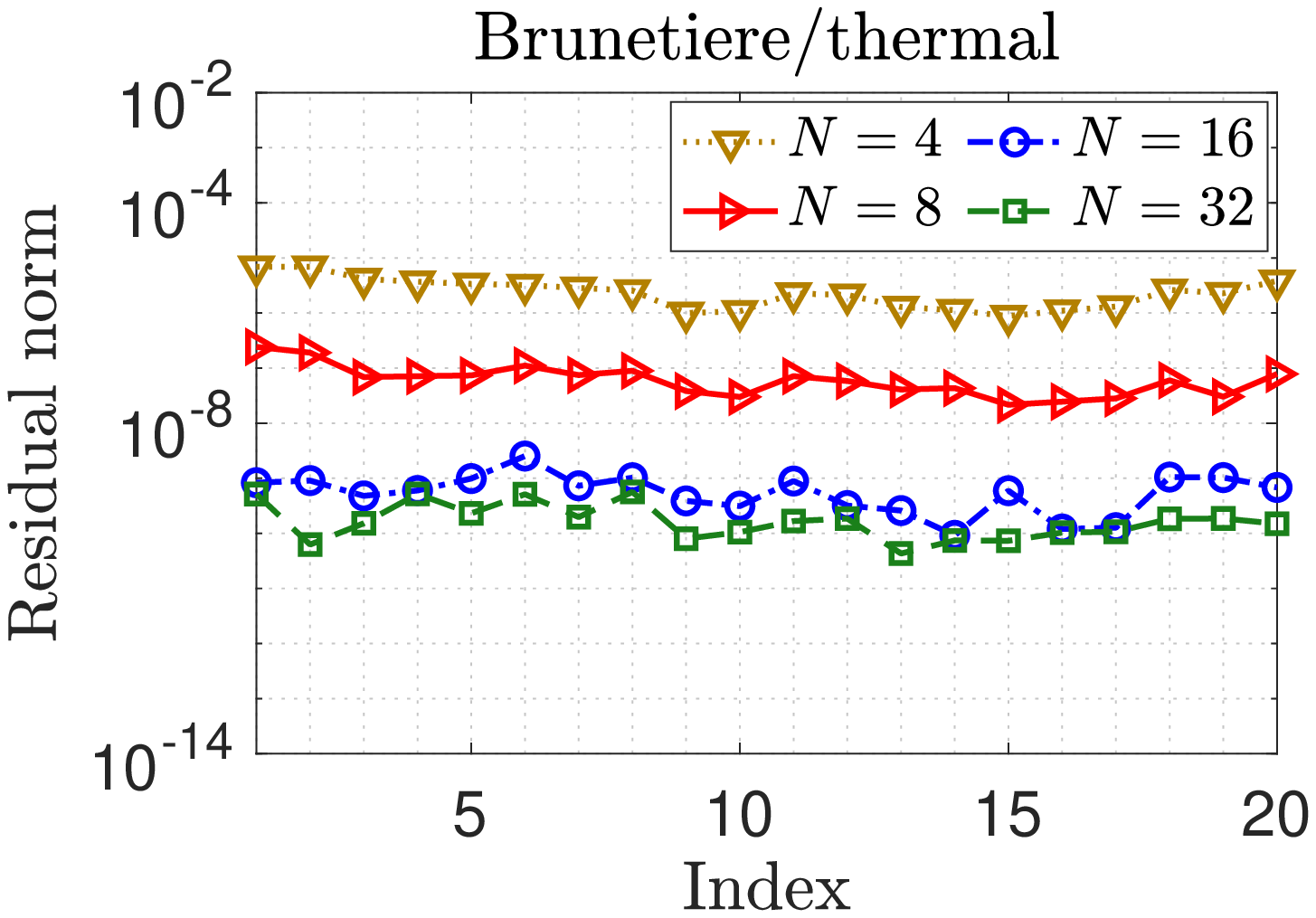}
\caption{Application of Algorithm \ref{alg:0} to matrices 
\texttt{wang1} (left column) and \texttt{thermal} (right column). 
First row: the $n_{ev}$ sought eigenvalues of $(A,M)$ and 
$n_{ev}$ immediate unwanted eigenvalues of smallest modulus. Second 
row: the ratio of smallest to largest singular value of matrix $G$ 
as determined at each iteration of Algorithm \ref{alg:2} during its 
application to matrix $\sum_{j=1}^{N}\omega_j S(\zeta_j)^{-1}$.
Third row: absolute eigenvalue error. Fourth row: residual norm. 
The indices of the $x$-axis are organized such that index ‘$i$' 
corresponds to the sought eigenvalue with the $i$th smallest real part. \label{wang1a}} 
\end{figure}
We consider the computation of the $n_{ev}=20$ eigenvalues of 
smallest modulus (and their associated eigenvectors) of the matrices 
\texttt{wang1} and \texttt{thermal}. The size of the Schur 
complement matrices after the application of the graph partitioner 
is equal to $s=576$, and $s=668$, respectively. 
The application of Algorithm \ref{alg:0} is visualized in Figure \ref{wang1a}. The first row 
of plots shows the $n_{ev}$ sought and $n_{ev}$ immediate unwanted eigenvalues of smallest 
modulus, while the second row of shows the ratio of the smallest to the largest singular value 
as determined at each iteration of Algorithm \ref{alg:2} during its application to matrix 
$\sum_{j=1}^{N}\omega_j S(\zeta_j)^{-1}$.\footnote{Results 
for matrix $\sum_{j=1}^{N} \omega_jB(\zeta_j)^{-1}F(\zeta_j)S(\zeta_j)^{-1}$ were essentially 
identical and thus not reported.} Notice that this ratio approximates zero and decreases faster 
for larger values of $N$ as a consequence of the fact that the rank of the matrix 
$\sum_{j=1}^{N}\omega_j S(\zeta_j)^{-1}$ approaches that 
of the matrix $\sum_{i=1}^{n_{ev}}y^{(i)}\left(\hat{y}^{(i)}\right)^H$, which 
is bounded from above by $n_{ev}$. Increasing the value of $N$ does not always lead to a proportional 
gain in terms of convergence rate. For example, increasing $N=8$ to $N=16$ reduces the number of 
iterations by a factor of at least four for the first two matrices considered. On the other hand, 
increasing $N=16$ to $N=32$ reduces the number of iterations by a factor which is smaller than two. 
Moreover, small values of $N$, e.g., $N=4$, might lead to very slow convergence. 
The third and fourth rows of plots show the associated eigenvalue errors (left column) and residual 
norms (right column) returned by Algorithm \ref{alg:0}. For the choice $N=4$, the range of matrices 
$\sum_{j=1}^{N} \omega_jB(\zeta_j)^{-1}F(\zeta_j)S(\zeta_j)^{-1}$ and 
$\sum_{j=1}^{N} \omega_jS(\zeta_j)^{-1}$ was not captured to high precision and 
this is reflected in the approximation of the sought eigenpairs. For the choices 
$N=8,\ 16$ and $N=32$, the range of both matrices was captured up to the required tolerance and the 
sought eigenpairs were captured to higher accuracy.

\begin{figure} 
\centering
\includegraphics[width=0.49\textwidth]{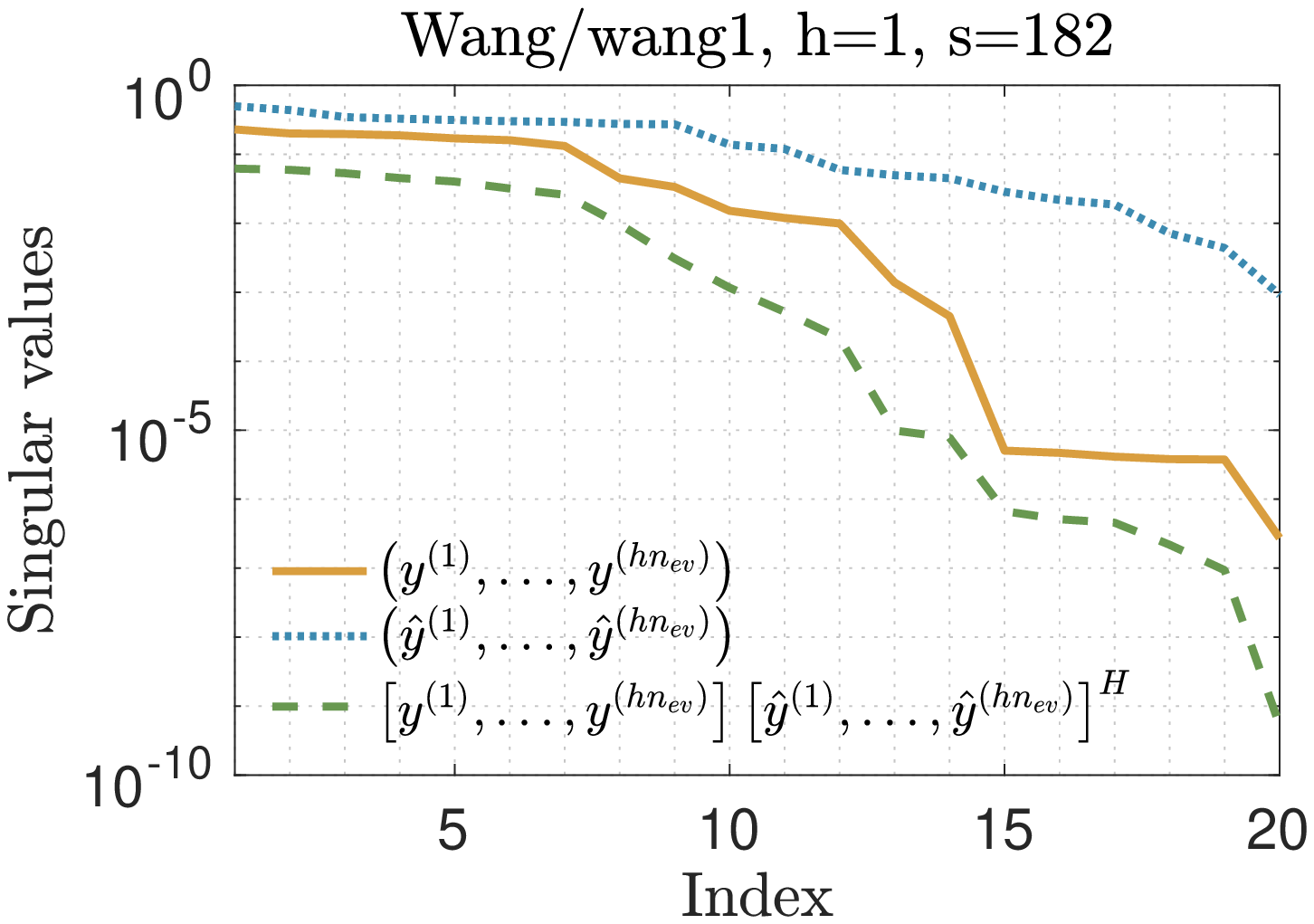}
\includegraphics[width=0.49\textwidth]{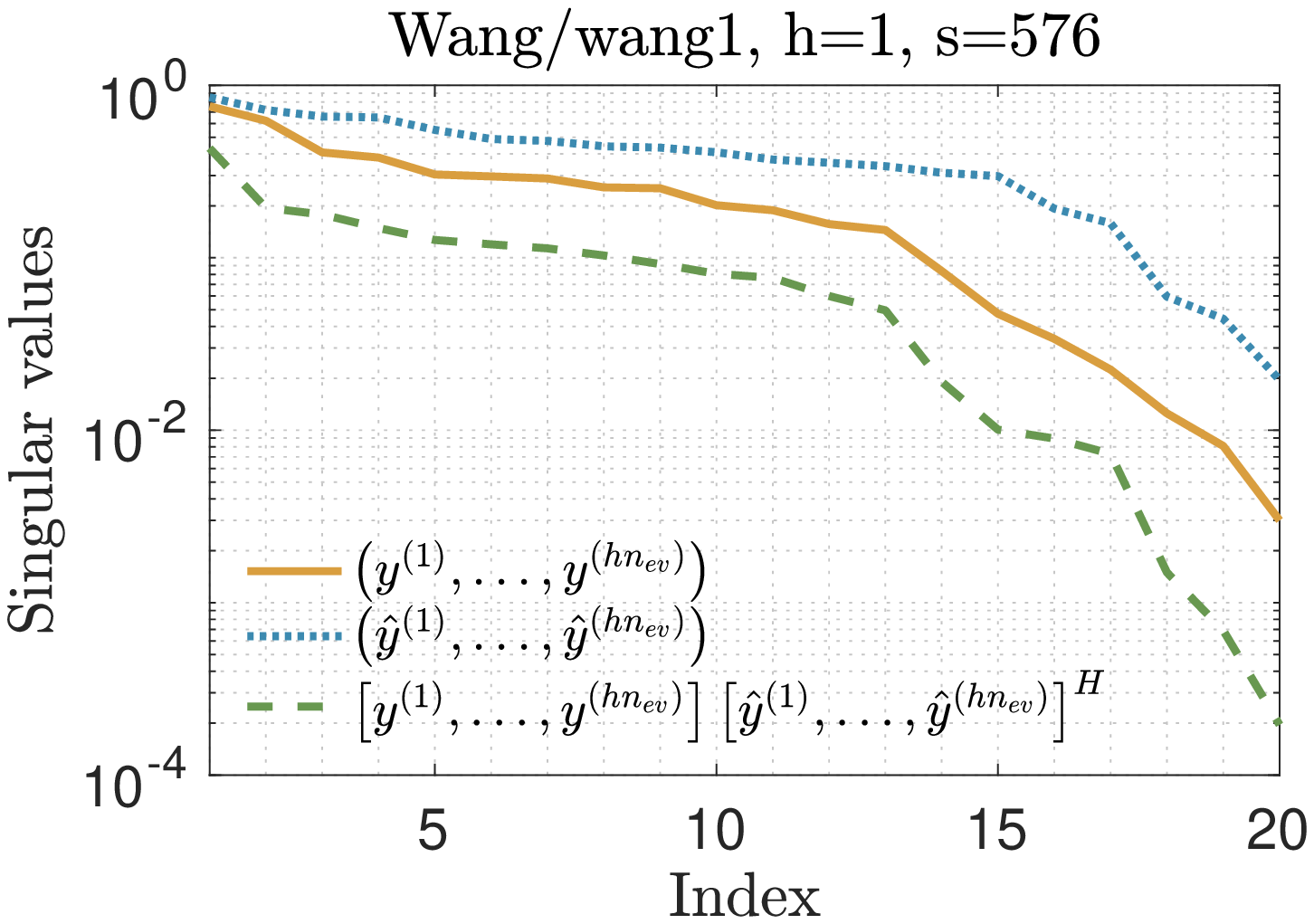}
\includegraphics[width=0.49\textwidth]{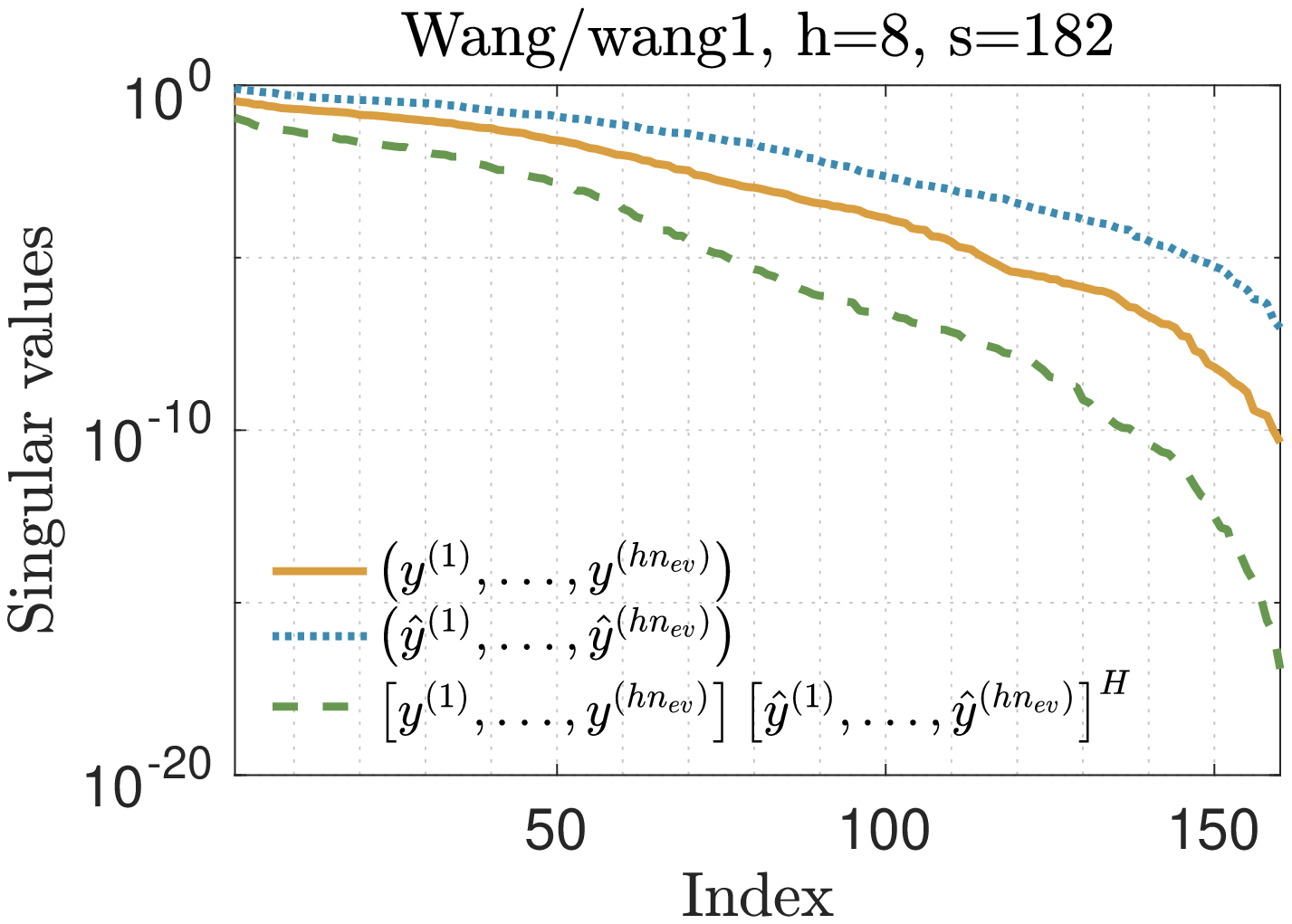}
\includegraphics[width=0.49\textwidth]{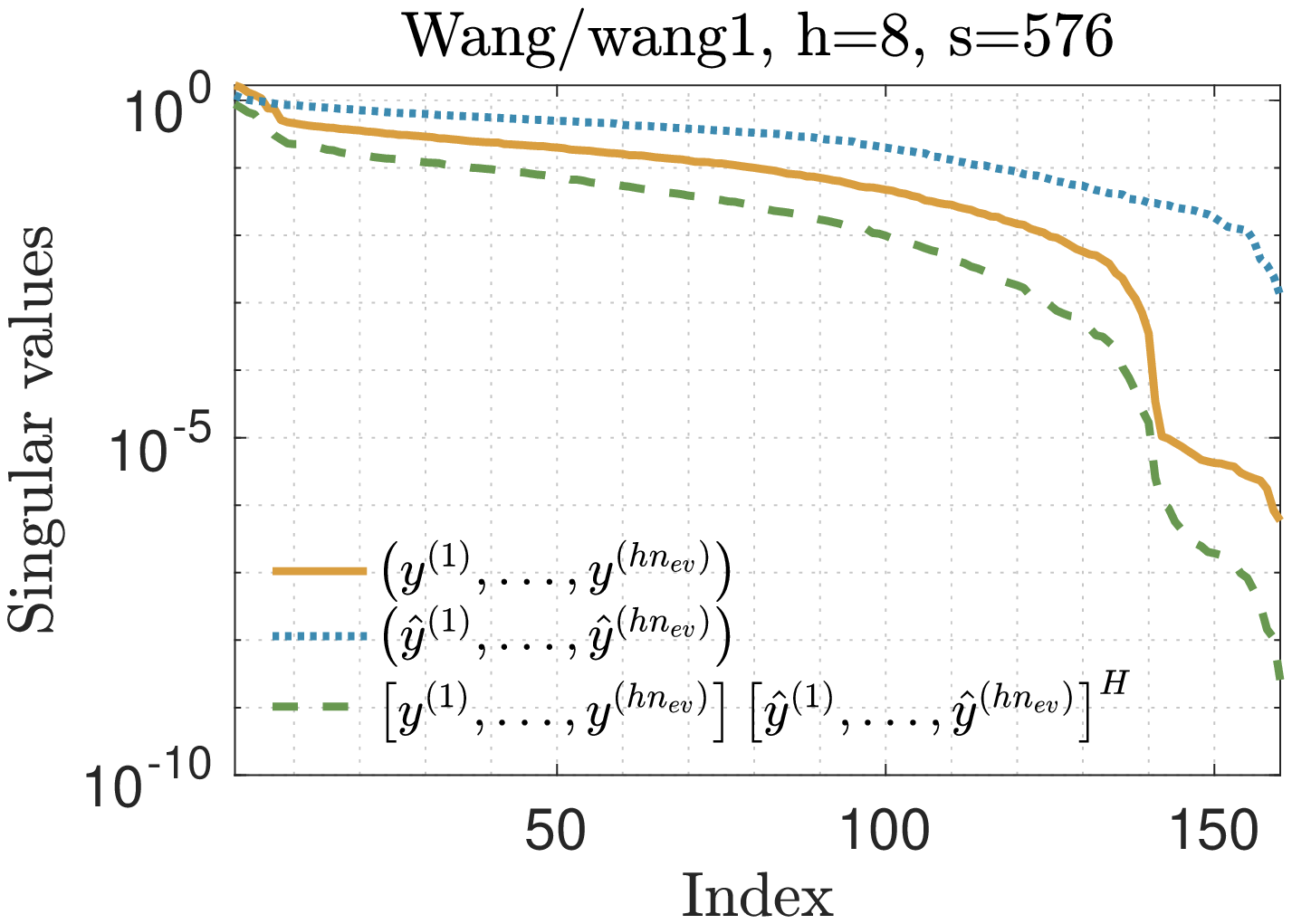}
\includegraphics[width=0.49\textwidth]{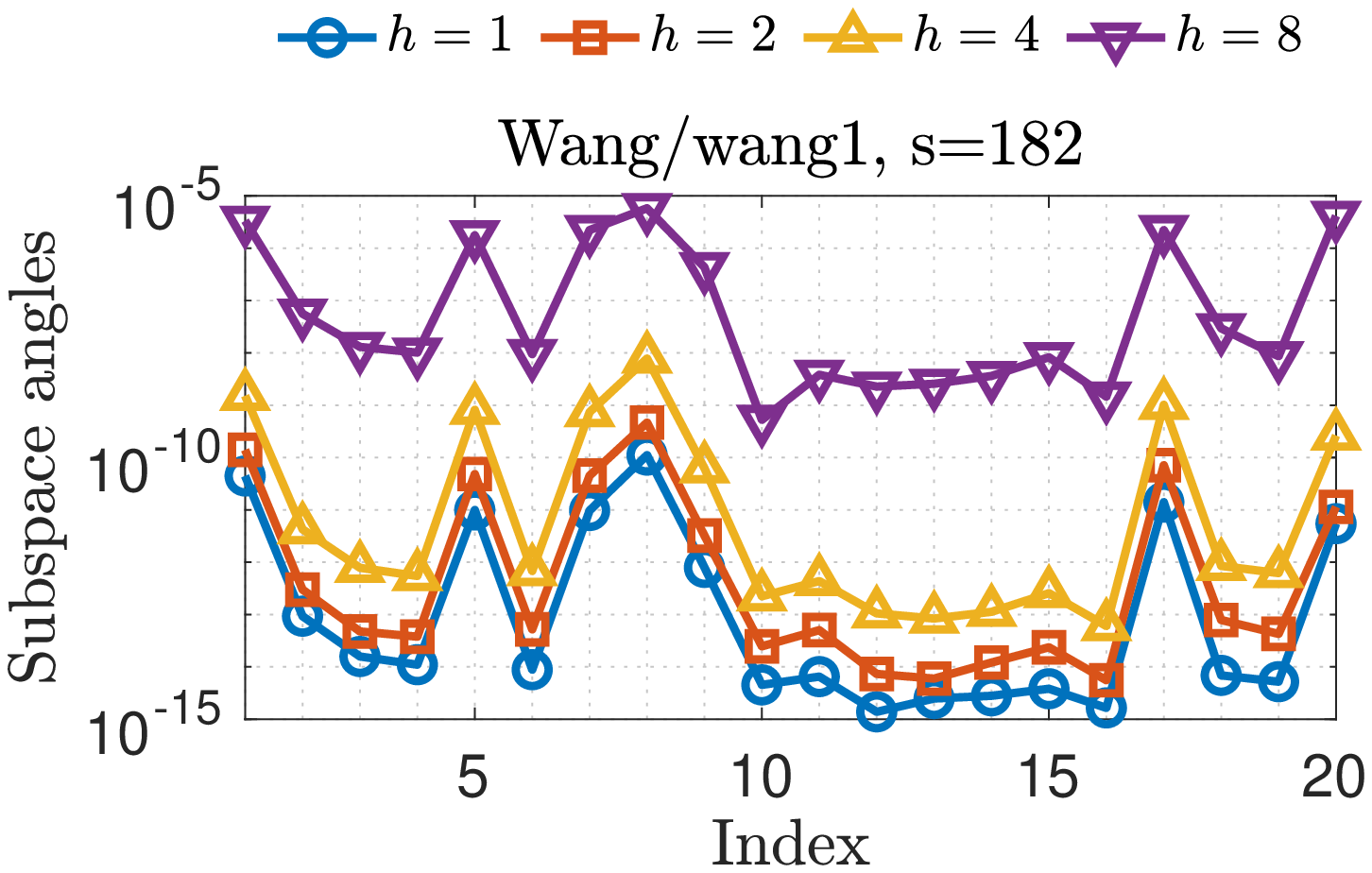}
\includegraphics[width=0.49\textwidth]{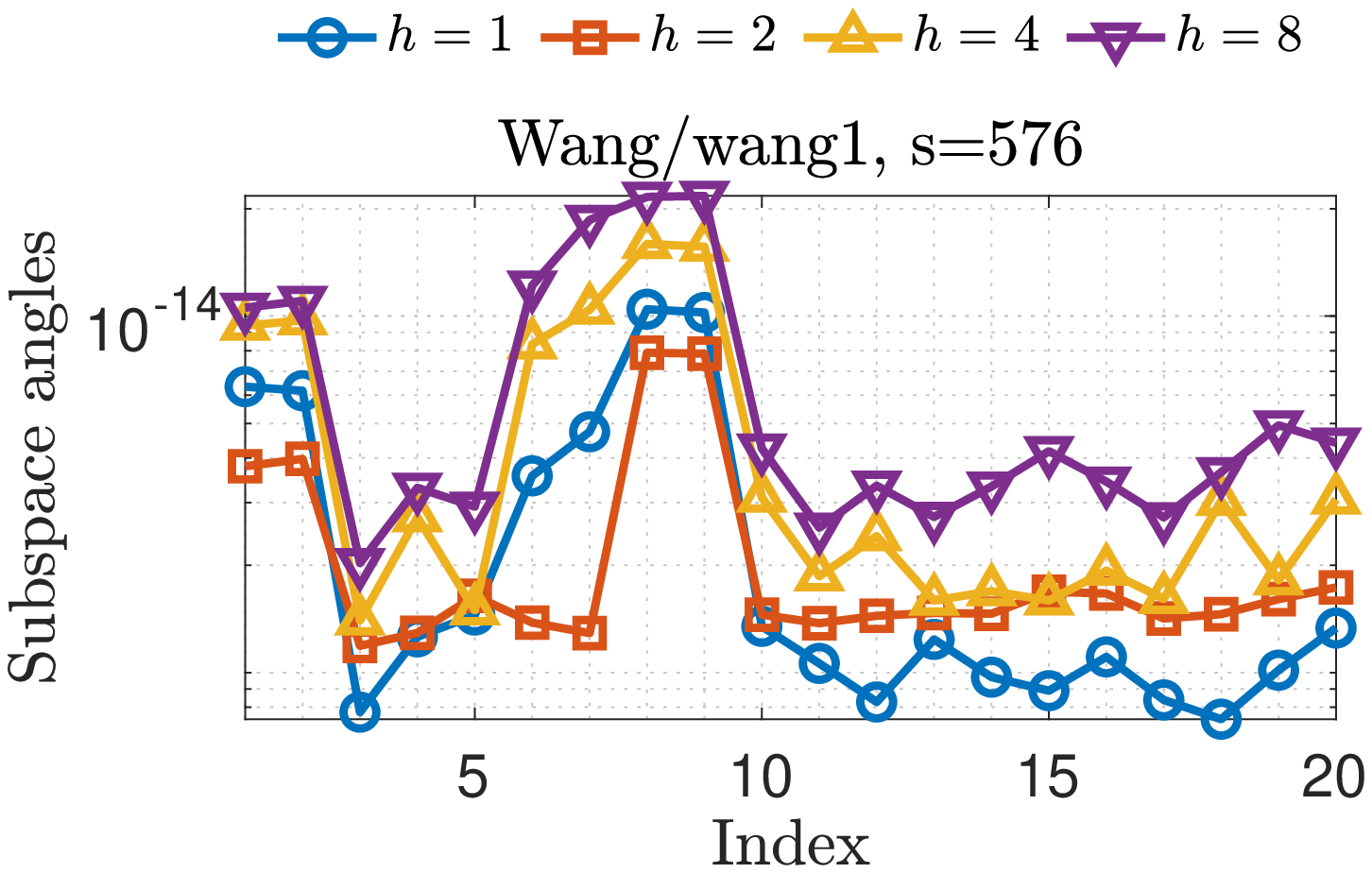}
\caption{{Singular values of matrices $\left[y^{(1)},\ldots,y^{(hn_{ev})}\right]$, 
$\left[\hat{y}^{(1)},\ldots,\hat{y}^{(hn_{ev})}\right]$, and their product
$\left[y^{(1)},\ldots,y^{(hn_{ev})}\right]\left[\hat{y}^{(1)},\ldots,\hat{y}^{(hn_{ev})}
\right]^H$, for different values of $h\in \mathbb{N}$ and size $s$ of the matrix 
$\sum_{j=1}^{N}\omega_j S(\zeta_j)^{-1}$. In all figures we set $n_{ev}=20$. 
First row: $h=1$ and matrix $A$ was partitioned so that the size of each Schur 
complement matrix in $\sum_{j=1}^{N}\omega_j S(\zeta_j)^{-1}$ is equal to $s=182$ (left) 
and $s=576$ (right). Second row: same as before but now we set $h=8$. Third row: 
angles between 
$\texttt{range}\left(\left[y^{(1)},\ldots,y^{(hn_{ev})}\right]
\left[\hat{y}^{(1)},\ldots,\hat{y}^{(hn_{ev})}\right]^H\right)$ 
and vectors $y^{(1)},\ldots,y^{(n_{ev})}$.}} \label{wang1c}
\end{figure}

The results in Figure \ref{wang1a} suggest that larger values of $N$ can lead to 
higher accuracy in Algorithm \ref{alg:0} even though the range of matrices 
$\sum_{j=1}^{N} \omega_jB(\zeta_j)^{-1}F(\zeta_j)S(\zeta_j)^{-1}$ and 
$\sum_{j=1}^{N} \omega_jS(\zeta_j)^{-1}$ is captured highly accurately for all 
$N=8,\ N=16$, and $N=32$. Consider for example the matrix $\sum_{j=1}^{N} \omega_jS(\zeta_j)^{-1}=\left[\rho(\lambda_i)y^{(i)}\right]_{\rho(\lambda_i)\neq 0}
\left[\hat{y}^{(i)}\right]^{H}_{\rho(\lambda_i)\neq 0}$. In practice, even 
though the condition in Proposition \ref{pro2} holds, some of the trailing 
non-zero singular values of the matrix $\sum_{j=1}^{N} \omega_jS(\zeta_j)^{-1}$ 
might be (much) smaller than those of $\left[\rho(\lambda_i)y^{(i)}\right]_{\rho(\lambda_i)\neq 0}$, thus “suppressing" some directions of 
$\texttt{span}\left(\left[\rho(\lambda_i)y^{(i)}\right]_{\rho(\lambda_i)\neq 0}\right)$ 
in the range of the matrix $\sum_{j=1}^{N} \omega_jS(\zeta_j)^{-1}$. 
Since these directions generally have a nonzero projection to the subspace 
$\texttt{span}\left(y^{(1)},\ldots,y^{(n_{ev})}\right)$, we expect that the
accuracy to which we can capture the latter subspace might also be reduced. 
The same holds for $\texttt{span}\left(u^{(1)},\ldots,u^{(n_{ev})}\right)$ and 
the range of matrix 
$\sum_{j=1}^{N} \omega_jB(\zeta_j)^{-1}F(\zeta_j)S(\zeta_j)^{-1}$ as well.

Figure \ref{wang1c} visualizes the above discussion for matrix \texttt{wang1}. 
In particular, we plot the singular values of the matrices $\left[y^{(1)},
\ldots,y^{(hn_{ev})}\right]$, $\left[\hat{y}^{(1)},\ldots,\hat{y}^{(hn_{ev})}\right]$, 
as well as those of their matrix product, for $h=1$ and $h=8$, and $n_{ev}=20$. 
Smaller values of $h$ simulate larger values of $N$. The size of the Schur 
complement matrices was varied as $s=182$ and $s=576$. Observe that 
the singular values of the matrix $\left[y^{(1)},\ldots,y^{(hn_{ev})}\right]
\left[\hat{y}^{(1)},\ldots,\hat{y}^{(hn_{ev})}\right]^H$ trail those of 
$\left[y^{(1)},\ldots,y^{(hn_{ev})}\right]$. As a result, 
some directions of $\texttt{span}\left(y^{(1)},\ldots,y^{(hn_{ev})}\right)$ are 
captured less accurately in 
$\mathtt{range}\left(\left[y^{(1)},\ldots,y^{(hn_{ev})}\right]\left[\hat{y}^{(1)},\ldots,
\hat{y}^{(hn_{ev})}\right]^H\right)$. The latter effect is sketched in the 
bottom row of plots where we plot the angle between the vectors $y^{(1)},\ldots,y^{(n_{ev})}$ and the 
range of matrix $\mathtt{range}\left(\left[y^{(1)},\ldots,y^{(hn_{ev})}\right]\left[\hat{y}^{(1)},
\ldots,\hat{y}^{(hn_{ev})}\right]^H\right)$. Ideally, all angles should be equal 
to zero. However, we observe larger angles when $h$ is larger (i.e., $N$ gets 
smaller) and the vectors $y^{(i)}$ and $\hat{y}^{(i)}$ lie in a lower-dimensional 
subspace (i.e., $s$ gets smaller).

\begin{figure*}[!htbp]
\centering
\includegraphics[width=0.49\textwidth]{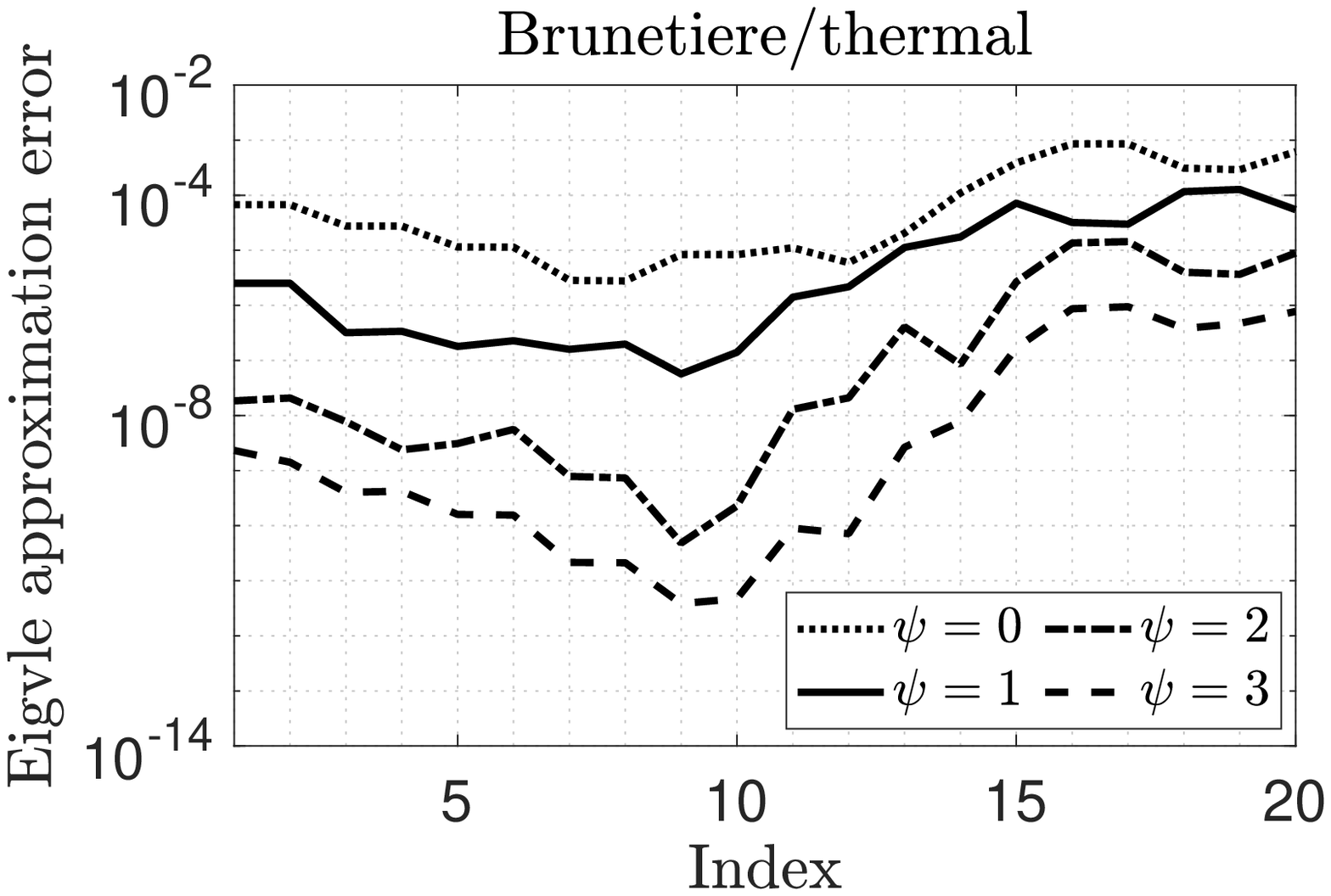}
\includegraphics[width=0.49\textwidth]{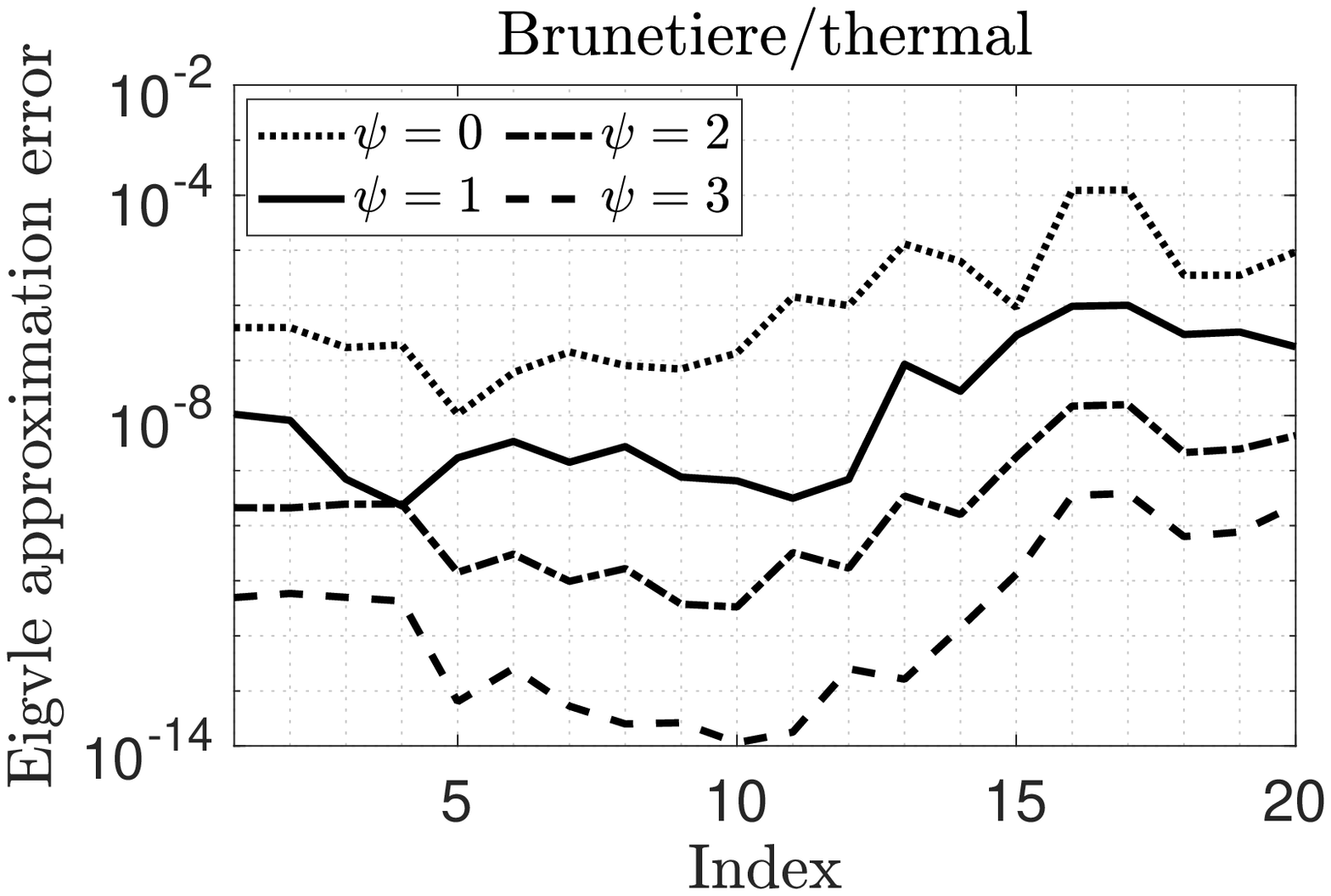}
\includegraphics[width=0.49\textwidth]{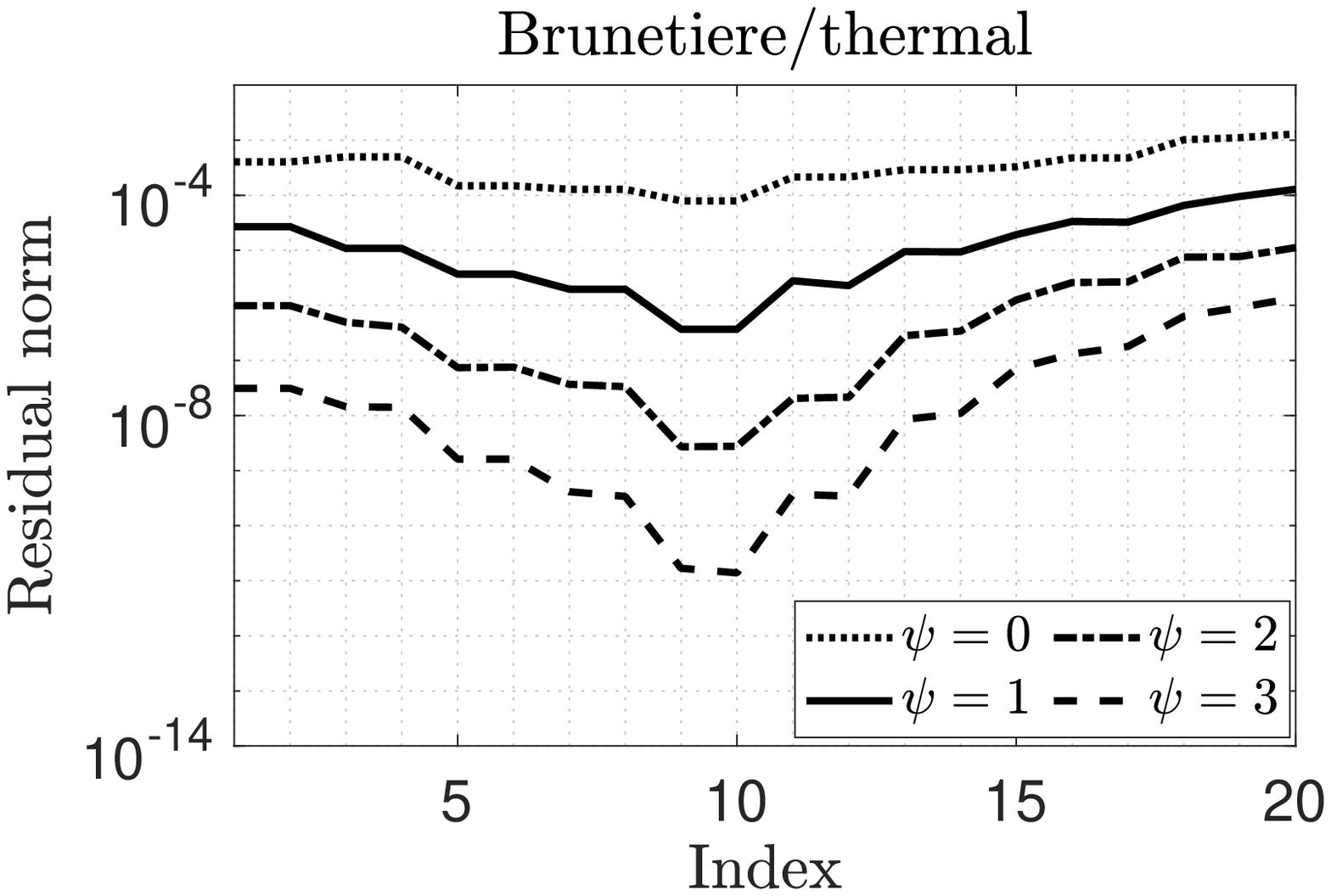}
\includegraphics[width=0.49\textwidth]{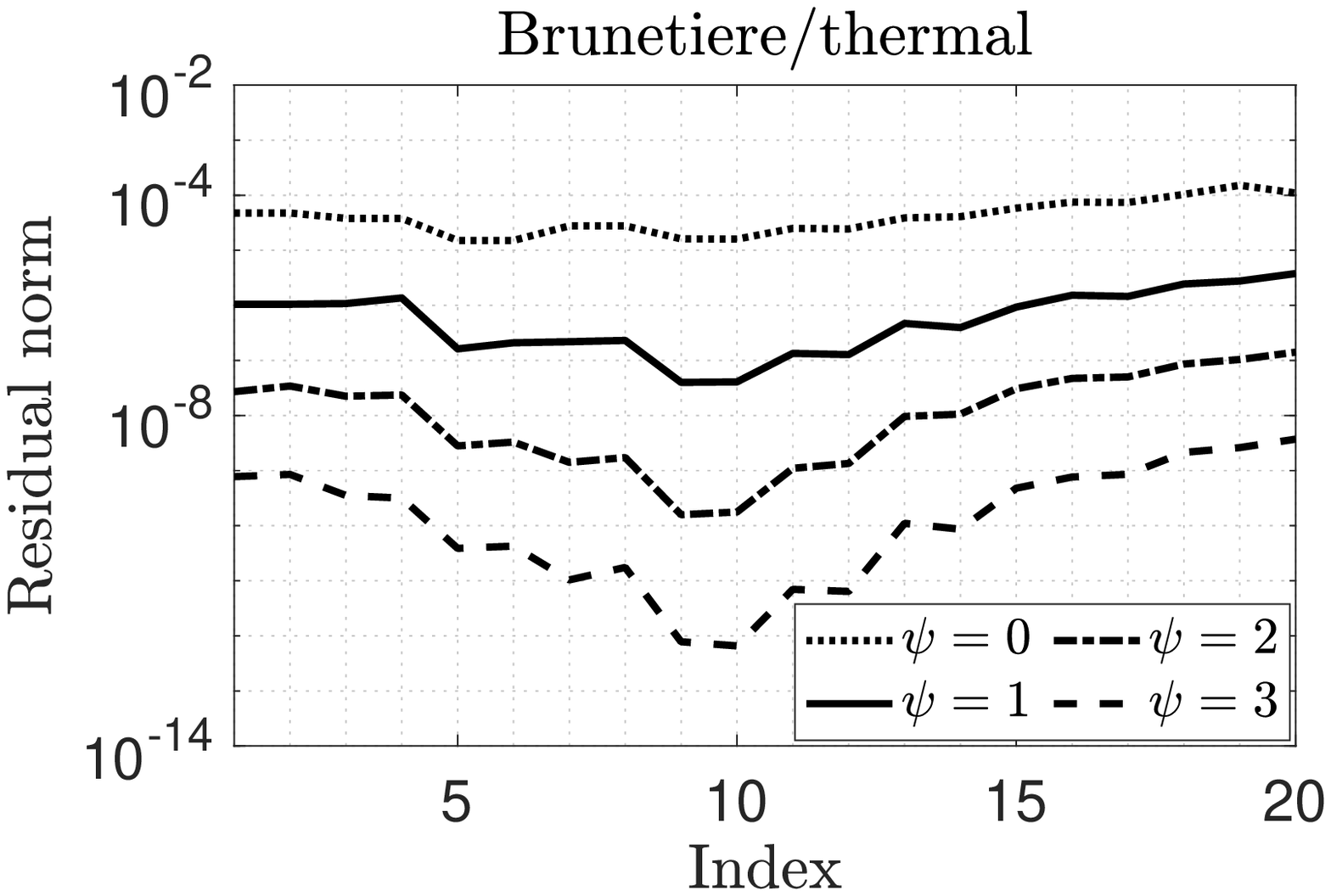}
\caption{Application of Algorithm \ref{alg:1} to the \texttt{thermal} matrix. 
Top row: absolute eigenvalue errors for different values of $\psi$ 
and $\phi=14$ (left) and $\phi=160$ (right). Bottom row: residual norms for 
the same values. The indices of the $x$-axis are organized such that index 
‘$i$' corresponds to the sought eigenvalue with the $i$th smallest real part. 
\label{thermal_alg2}} 
\end{figure*}
Figure \ref{thermal_alg2} plots the eigenvalue approximation errors and corresponding 
residual norms in the approximation of the $n_{ev}$ smallest modulus eigenvalues 
and associated eigenvectors of the \texttt{thermal} matrix by Algorithm \ref{alg:1}. 
The number of computed matrix resolvent terms was set to $\psi=0,\ 1,\ 2$, and $\psi=3$, 
while the number of computed (deflated) eigenvectors was varied to $\phi=14$ and $\phi=160$. 
The choice $\phi=14$ coincides with computing only those eigenvalues of the pencil 
$(B,M_B)$ located inside the disk ${\cal D}$. The number of poles was set equal to 
$N=16$.\footnote{The results obtained for the choices $N=8$ and $N=32$ were essentially 
identical to those for the case $N=16$, and thus are not reported.} As expected, the 
accuracy in the approximation of the sought eigenpairs improves with larger values of 
$\psi$ since the action of the matrix $\left(I-V_\phi \hat{V}_\phi^HM_B\right)B(\lambda_i)^{-1}$ 
is now  approximated more accurately. Similarly, increasing $\phi=14$ to $\phi=160$ 
leads to enhanced accuracy.  {Note that the major improvements in 
accuracy come from increasing the value of $\psi$. This was a general trend observed 
for the rest of our test matrices as well.} Moreover, as was also discussed in Section 
\ref{second_alg}, the accuracy obtained by Algorithm \ref{alg:1} depends on the location 
of each eigenvalue $\lambda_i \in {\cal D}$ since the action of $\left(I-V_\phi \hat{V}_\phi^HM_B\right)B(\lambda_i)^{-1}$ is better approximated for those $\lambda_i \in {\cal D}$ located closer to the center of the disk ${\cal D}$. This is in contrast to Algorithm \ref{alg:0} which provides an almost uniform accuracy for all eigenpairs $(\lambda_i,x^{(i)})$ for which $\lambda_i \in {\cal D}$.

\subsection{Comparisons against subspace iteration with rational filtering}
We now consider the computation of the $n_{ev}=40$ eigenvalues of smallest 
modulus (and their associated eigenvectors) of the matrix pencil 
\texttt{bfw782}, and the matrices \texttt{utm1700b}, \texttt{rdb3200l}, 
\texttt{dw4096}, and \texttt{big}. 
Figure \ref{eigvals} plots the $2n_{ev}$ eigenvalues of smallest modulus 
of the last four matrices. 
\begin{figure*}[!htbp]
\centering
\includegraphics[width=0.49\textwidth]{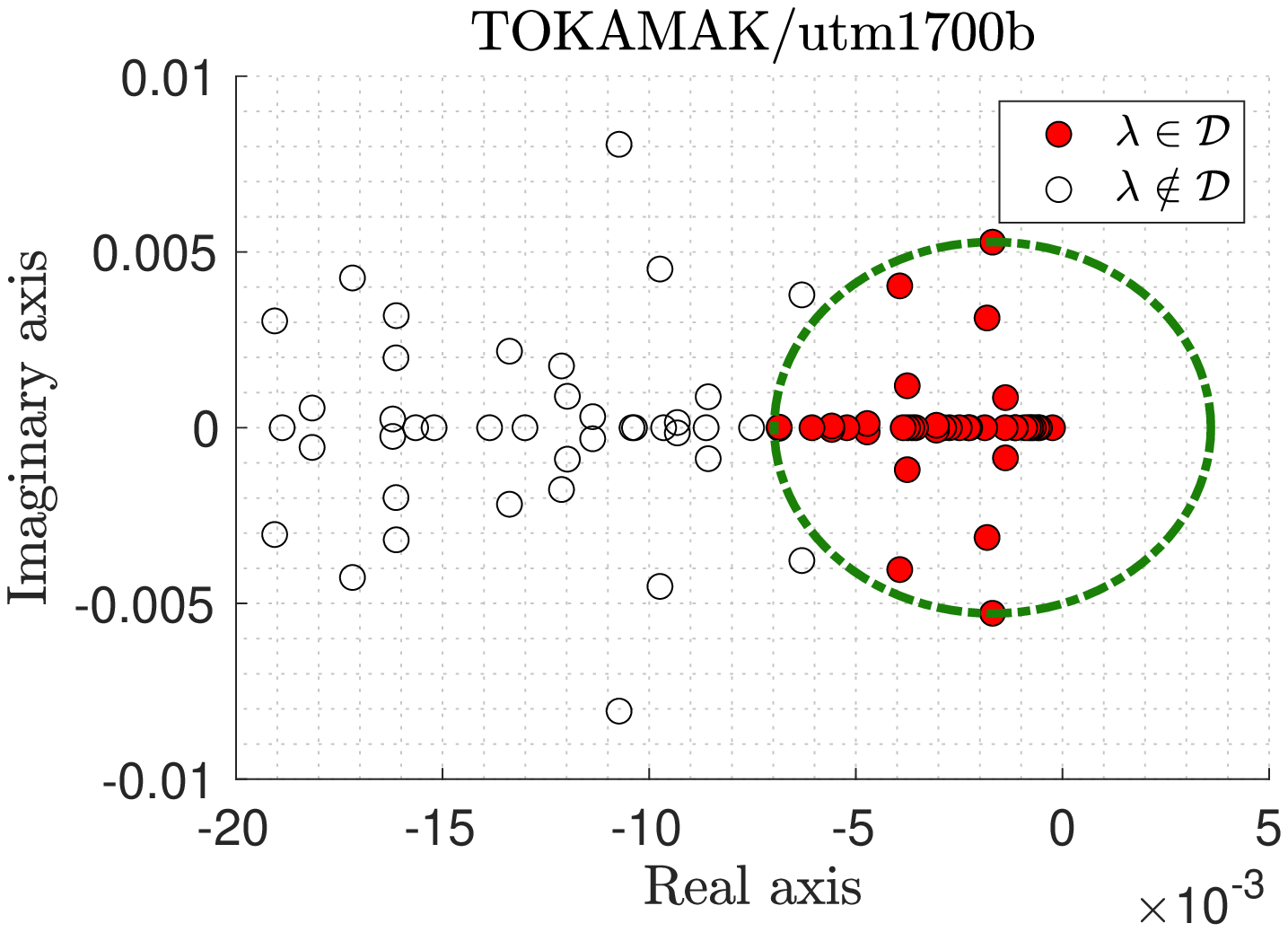}
\includegraphics[width=0.49\textwidth]{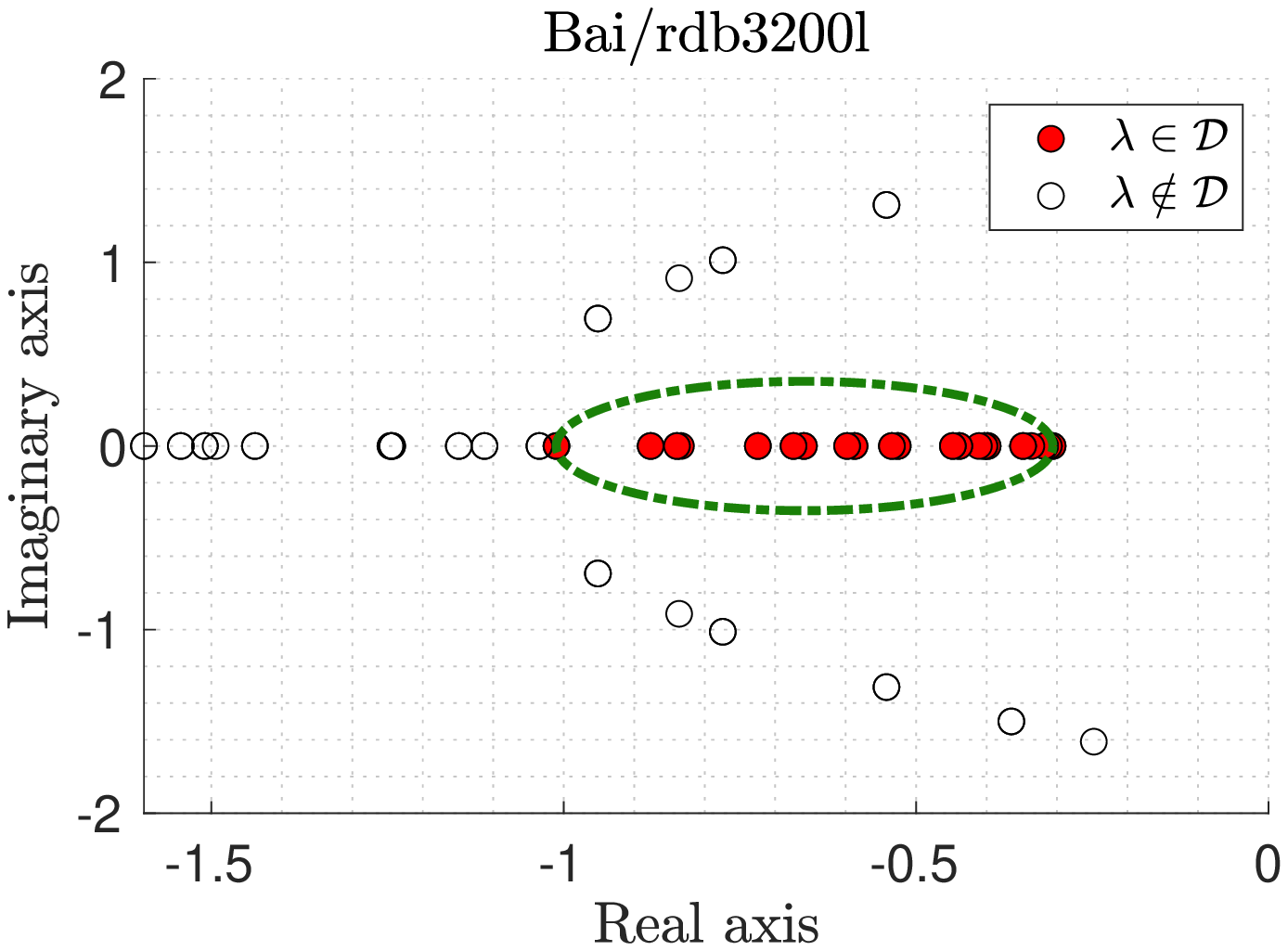}
\includegraphics[width=0.49\textwidth]{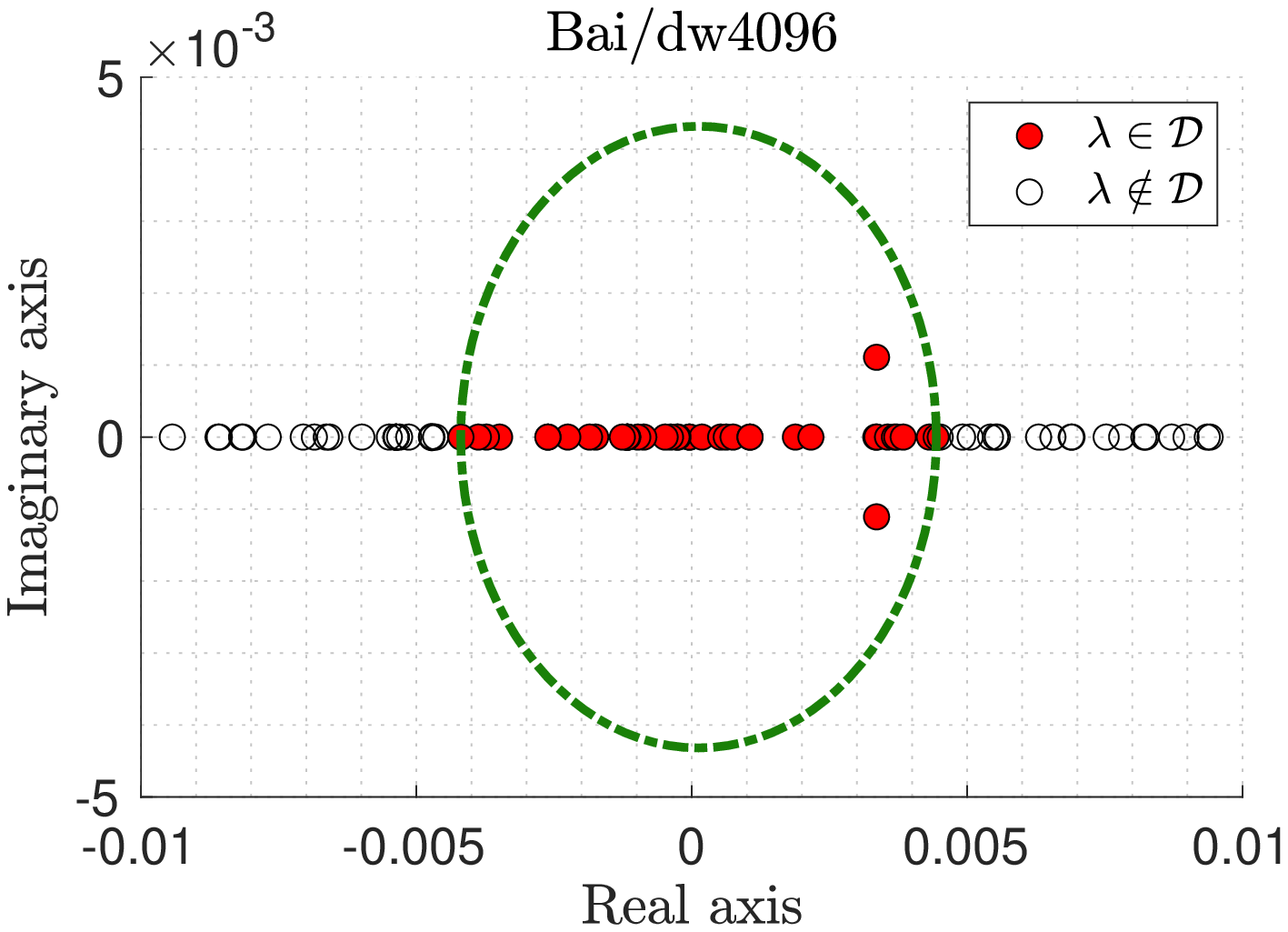}
\includegraphics[width=0.49\textwidth]{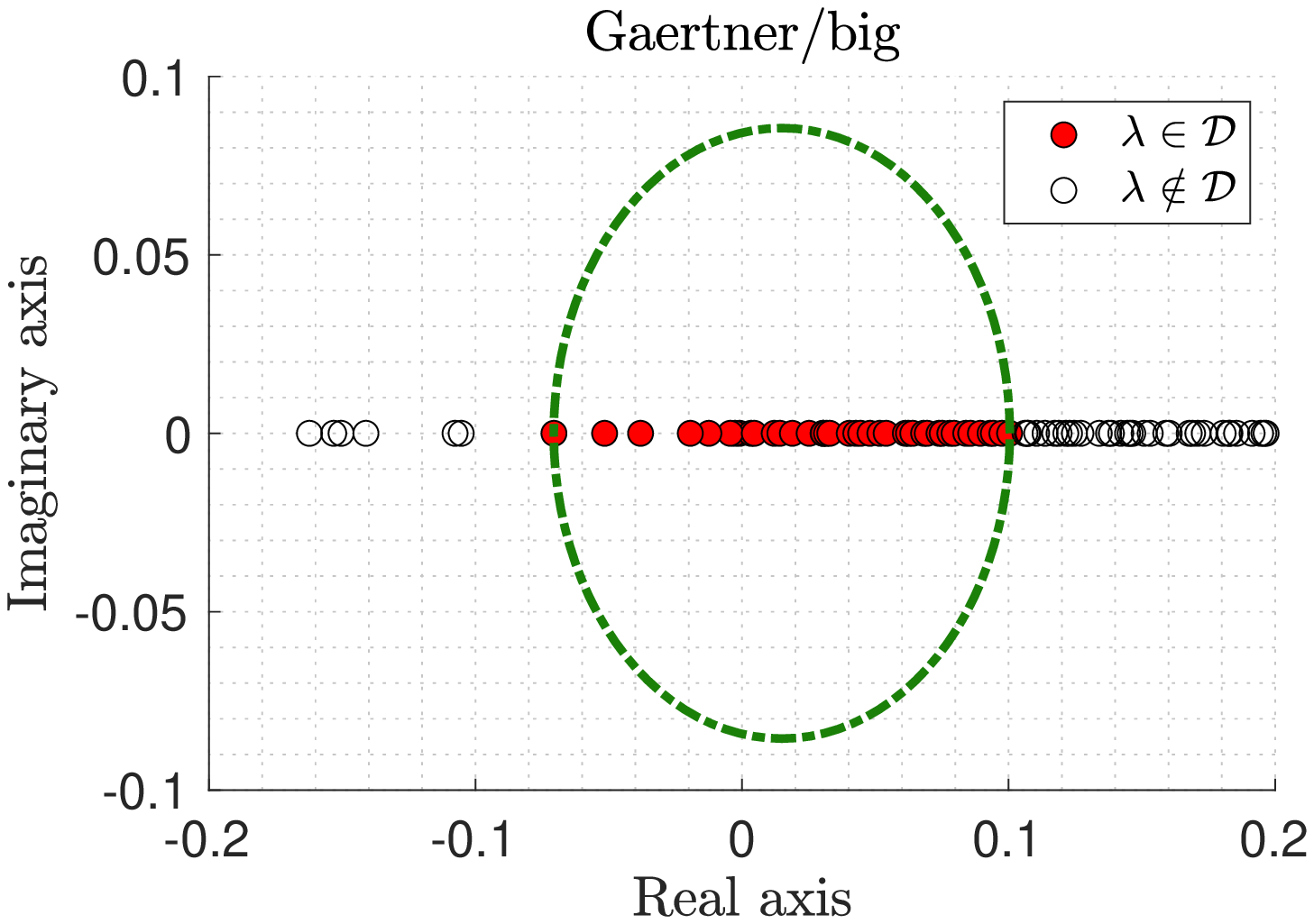}
\caption{Plot of the $2n_{ev}$ eigenvalues of smallest modulus of 
some of the matrices listed in Table \ref{testmat}.\label{eigvals}} 
\end{figure*}
\begin{table*}[th!]
\centering
\caption{Minimum and maximum eigenvalue errors and residual norms 
returned by Algorithm \ref{alg:0} for the test matrices listed in 
Table \ref{testmat}. The total number of iterations performed by 
Algorithm \ref{alg:0} is also reported. A value ‘F' indicates that 
the loop in Algorithm \ref{alg:0} did not terminate within four 
hundred iterations.}
	{\scriptsize \begin{tabular}{clcccc|c}
			\toprule \midrule
			& {\rm } & \multicolumn{5}{c}{Algorithm \ref{alg:0}}  \\
			\cmidrule(lr){3-7}
			& {} & $\min |\lambda-\hat{\lambda}|$ & $\max |\lambda-\hat{\lambda}|$ & $\min \hat{\rho}$ & $\max \hat{\rho}$ & {\rm It} \\
			\midrule
			\parbox[t]{2mm}{\multirow{4}{*}{\rotatebox[origin=c]{90}{\texttt{bfw782}}}} 
			& $\boldsymbol{N=4}$    & $4.3 \times 10^{-1}$  & $2.1 \times 10^{-0}$  & $3.2 \times 10^{-2}$ & $2.4 \times 10^{-1}$ & 87  \\
			& $\boldsymbol{N=8}$    & $1.5 \times 10^{-2}$  & $3.2 \times 10^{-1}$  & $2.9 \times 10^{-4}$ & $2.0 \times 10^{-2}$ & 87   \\
			& $\boldsymbol{N=16}$  & $1.3 \times 10^{-7}$  & $4.5 \times 10^{-4}$  & $5.9 \times 10^{-8}$ & $1.0 \times 10^{-6}$ & 76   \\
			& $\boldsymbol{N=32}$  & $1.1 \times 10^{-9}$  & $7.4 \times 10^{-6}$  & $6.6 \times 10^{-10}$ & $2.8 \times 10^{-8}$& 55   \\   
			\midrule
			\parbox[t]{2mm}{\multirow{4}{*}{\rotatebox[origin=c]{90}{\texttt{utm1700b}}}} 
			& $\boldsymbol{N=4}$   & $1.4 \times 10^{-8}$  & $1.3 \times 10^{-4}$  & $2.0 \times 10^{-7}$ & $3.9 \times 10^{-6}$ & 235  \\
			& $\boldsymbol{N=8}$   & $1.2 \times 10^{-9}$  & $3.2 \times 10^{-6}$  & $6.9 \times 10^{-9}$ & $2.5 \times 10^{-7}$ & 134  \\
			& $\boldsymbol{N=16}$  & $8.0 \times 10^{-12}$  & $7.0 \times 10^{-8}$  & $5.9 \times 10^{-10}$ & $4.0 \times 10^{-8}$ & 72  \\
			& $\boldsymbol{N=32}$  & $1.6 \times 10^{-11}$  & $3.1 \times 10^{-8}$  & $6.6 \times 10^{-10}$ & $6.0 \times 10^{-8}$& 53   \\   
			\midrule
			\parbox[t]{2mm}{\multirow{4}{*}{\rotatebox[origin=c]{90}{\texttt{rdb3200l}}}} 
			& $\boldsymbol{N=4}$   & $2.0 \times 10^{-5}$  & $1.7 \times 10^{-3}$  & $9.1 \times 10^{-4}$ & $1.5 \times 10^{-2}$ & 296  \\
			& $\boldsymbol{N=8}$   & $7.2 \times 10^{-9}$  & $5.7 \times 10^{-4}$  & $2.3 \times 10^{-7}$ & $3.5 \times 10^{-6}$ & 161  \\
			& $\boldsymbol{N=16}$  & $1.6 \times 10^{-12}$  & $3.9 \times 10^{-9}$  & $7.5 \times 10^{-10}$ & $4.5 \times 10^{-8}$ & 77  \\
			& $\boldsymbol{N=32}$  & $1.3 \times 10^{-15}$  & $2.3 \times 10^{-13}$  & $2.1 \times 10^{-11}$ & $5.6 \times 10^{-10}$& 52  \\   
			\midrule
			\parbox[t]{2mm}{\multirow{4}{*}{\rotatebox[origin=c]{90}{\texttt{dw4096}}}} 
			& $\boldsymbol{N=4}$   & $2.4 \times 10^{-8}$   & $4.0 \times 10^{-5}$  & $2.9 \times 10^{-5}$  & $1.4 \times 10^{-2}$  & F    \\
			& $\boldsymbol{N=8}$   & $7.2 \times 10^{-12}$  & $6.1 \times 10^{-9}$   & $1.9 \times 10^{-8}$  & $3.1 \times 10^{-5}$ & 329  \\
			& $\boldsymbol{N=16}$  & $2.6 \times 10^{-15}$  & $1.8 \times 10^{-10}$  & $2.7 \times 10^{-10}$ & $5.5 \times 10^{-6}$ & 147  \\
			& $\boldsymbol{N=32}$  & $9.8 \times 10^{-15}$  & $3.4 \times 10^{-13}$  & $2.6 \times 10^{-12}$ & $1.5 \times 10^{-11}$& 75  \\   
			\midrule
			\parbox[t]{2mm}{\multirow{4}{*}{\rotatebox[origin=c]{90}{\texttt{big}}}} 
			& $\boldsymbol{N=4}$   & $4.8 \times 10^{-6}$  & $3.0 \times 10^{-2}$  & $3.8 \times 10^{-4}$ & $1.1 \times 10^{-2}$ & 377   \\
			& $\boldsymbol{N=8}$   & $1.5 \times 10^{-11}$  & $3.6 \times 10^{-6}$  & $2.1 \times 10^{-6}$ & $1.2 \times 10^{-4}$ & 226  \\
			& $\boldsymbol{N=16}$  & $6.2 \times 10^{-13}$  & $1.7 \times 10^{-9}$  & $1.0 \times 10^{-8}$ & $2.8 \times 10^{-6}$ & 108  \\
			& $\boldsymbol{N=32}$  & $1.3 \times 10^{-14}$  & $6.3 \times 10^{-11}$  & $1.8 \times 10^{-9}$ & $1.3 \times 10^{-6}$& 68  \\   
			\midrule
			\bottomrule	        		       				           			
	\end{tabular}}\label{tab:table1}
\end{table*}
Table \ref{tab:table1} lists the maximum and minimum absolute eigenvalue 
errors and associated residual norms returned by Algorithm \ref{alg:0} 
as $N$ varies. The number of iterations performed by Algorithm \ref{alg:0} 
is also listed. As expected, larger values of $N$ lead to fewer iterations 
since the rational filter $\rho(\zeta)$ decays faster outside ${\cal D}$. 
In agreement with the results discussed in Figure \ref{wang1c}, larger 
values of $N$ also lead to higher accuracy.
\begin{table*}[tbhp]
\caption{Maximum eigenvalue error and residual norm returned 
by Algorithm \ref{alg:1} for some of the test matrices listed 
in Table \ref{testmat}. These results were obtained by setting
$N=16$ and varying the value of computed resolvent terms $\psi$. } \label{tab:table3}
\setlength\tabcolsep{4pt}
\centering
{\scriptsize
	\begin{tabular}{ll|c|c|c|c|c}
			\toprule \midrule
			&  & \texttt{bfw782} & \texttt{utm1700b} & \texttt{rdb3200l} & \texttt{wd4096} & \texttt{big} \\
			\midrule
			\parbox[t]{2mm}{\multirow{4}{*}{\rotatebox[origin=c]{90}{$|\lambda-\hat{\lambda}|$}}} 
			& $\psi=0$    & $9.0 \times 10^{-2}$ & $5.1 \times 10^{-3}$ & $1.5 \times 10^{-1}$ & $2.8 \times 10^{-1}$ & $7.2 \times 10^{-2}$  \\
			& $\psi=1$    & $5.9 \times 10^{-3}$ & $1.0 \times 10^{-5}$ & $9.1 \times 10^{-3}$ & $1.2 \times 10^{-1}$ & $2.0 \times 10^{-3}$  \\
			& $\psi=2$    & $1.5 \times 10^{-4}$ & $2.8 \times 10^{-7}$ & $2.4 \times 10^{-4}$ & $2.9 \times 10^{-3}$ & $1.4 \times 10^{-5}$  \\
			& $\psi=3$    & $8.4 \times 10^{-6}$ & $9.0 \times 10^{-8}$ & $7.5 \times 10^{-6}$ & $9.7 \times 10^{-5}$ & $8.9 \times 10^{-7}$  \\
			\midrule
			\parbox[t]{2mm}{\multirow{4}{*}{\rotatebox[origin=c]{90}{$\hat{\rho}$}}} 
			& $\psi=0$      & $4.3 \times 10^{-2}$ &   $1.2 \times 10^{-3}$ &   $1.0 \times 10^{-0}$ &  $4.6 \times 10^{-1}$& $2.6 \times 10^{-1}$ \\
			& $\psi=1$      & $4.1 \times 10^{-3}$ &   $4.7 \times 10^{-4}$ &   $6.0 \times 10^{-2}$ &  $4.3 \times 10^{-1}$& $1.1 \times 10^{-1}$ \\
			& $\psi=2$      & $2.4 \times 10^{-3}$ &   $2.8 \times 10^{-6}$ &   $1.6 \times 10^{-3}$ &  $3.0 \times 10^{-2}$& $4.0 \times 10^{-3}$ \\
			& $\psi=3$      & $3.8 \times 10^{-5}$ &   $6.6 \times 10^{-7}$ &   $7.3 \times 10^{-5}$ &  $1.2 \times 10^{-4}$& $6.2 \times 10^{-5}$ \\
			\midrule
			\bottomrule	        		       				           			
	\end{tabular}}
\end{table*}
Additionally, Table \ref{tab:table3} lists the maximum eigenvalue error and residual
norm returned by Algorithm \ref{alg:1} when the value of $\phi$ is set equal to the 
number of eigenvalues located inside the disk ${\cal D}$ and $\psi$ varies. 

\begin{table*}[th!]
\caption{Average number of linear systems of the form 
$B(\zeta_j)x_d=b_d$ and $S(\zeta_j)x_s=b_s$ solved per
pole $\zeta_j,\ j=1,\ldots,N$. For Algorithm \ref{alg:1} we consider 
only the case $\psi=3$. For RSI we set  
$m:=m_1=1.1n_{ev},\ m:=m_2=1.5n_{ev},$ and $m:=m_3=2n_{ev}$. }
\centering
\setlength\tabcolsep{3pt}
{\scriptsize
	\begin{tabular}{ll|c|c|c|c|c|c|c|c|c|c}
			\toprule \midrule
			& {\rm } & \multicolumn{2}{c}{\texttt{bfw782}} & \multicolumn{2}{c}{\texttt{utm1700b}} & \multicolumn{2}{c}{\texttt{rdb3200l}} & \multicolumn{2}{c}{\texttt{dw4096}} & \multicolumn{2}{c}{\texttt{big}}\\
			\cmidrule(lr){3-4}\cmidrule(lr){5-6}\cmidrule(lr){7-8}\cmidrule(lr){9-10}
			\cmidrule(lr){11-12}
			& {} & $B(\zeta_j)$ & $S(\zeta_j)$ & $B(\zeta_j)$ & $S(\zeta_j)$ & $B(\zeta_j)$ & $S(\zeta_j)$ & $B(\zeta_j)$ & $S(\zeta_j)$ & $B(\zeta_j)$ & $S(\zeta_j)$ \\
			\midrule
			\parbox[t]{2mm}{\multirow{4}{*}{\rotatebox[origin=c]{90}{$\boldsymbol{N=8}$}}} 
			& Alg. \ref{alg:0}     & 87 & 87 & 134 &134 & 161 & 161 & 329 & 329 & 226 &226 \\ 
			& Alg. \ref{alg:1}     & 112 & 87 & 76 & 134 & 71 & 161 & 176 & 329 & 127 & 226 \\  
			& RSI($m_1$)           & 264 &132 & 616&308 & 616& 308& 616& 308& 704&352\\ 
			& RSI($m_2$)           & 120 & 60&240 & 120& 240& 120& 480& 240& 480&240\\
			& RSI($m_3$)           & 160 & 80&320 & 160& 320& 160& 320& 160& 320&160\\
			\midrule
			\parbox[t]{2mm}{\multirow{4}{*}{\rotatebox[origin=c]{90}{$\boldsymbol{N=16}$}}} 
			& Alg. \ref{alg:0}     & 76 & 76 & 72 & 72 & 77 & 77 & 147& 147&108 &108\\  
			& Alg. \ref{alg:1}     & 50 & 76 & 23 & 72 & 26 & 77 & 42 & 147 & 33 & 108\\ 
			& RSI($m_1$)           & 440 & 220& 352& 176& 352& 176 & 616&308 & 352&176\\ 
			& RSI($m_2$)           & 360 & 180& 120& 60 &240 & 120& 240& 120& 240& 120\\
			& RSI($m_3$)           & 320 & 160& 160& 80& 160& 80& 160& 80& 160&80\\
			\midrule
			\parbox[t]{2mm}{\multirow{4}{*}{\rotatebox[origin=c]{90}{$\boldsymbol{N=32}$}}} 
			& Alg. \ref{alg:0}     & 55  & 55 & 53& 53& 52& 52&75 &75 & 68 & 68\\  
			& Alg. \ref{alg:1}     & 20 & 55 & 9 & 53& 10 & 52 & 12 & 75 & 12 & 68 \\ 
			& RSI($m_1$)           & 264 & 132& 176& 88& 264&132 & 616&308 & 352&176\\ 
			& RSI($m_2$)           & 240 & 120& 120& 60& 120& 60& 240& 120& 240&120\\
			& RSI($m_3$)           & 160 & 80& 160& 80& 160& 80& 160& 80& 160&80\\
			\midrule
			\bottomrule	        		       				           			
	\end{tabular}}\label{tab:table2}
\end{table*}

Table \ref{tab:table2} lists the average number (with respect to $N$) of 
linear systems  of the form $B(\zeta)x_d=b_d$ and $S(\zeta)x_s=b_s$ solved 
by Algorithm \ref{alg:0}, Algorithm \ref{alg:1}, and RSI.\footnote{{These 
numbers do not include the cost to compute a good approximation of $n_{ev}$ 
in RSI.}} The loop in RSI 
terminates when the maximum residual in the approximation of the $n_{ev}$ 
sought eigenpairs becomes smaller than or equal to the maximum residual norm 
achieved by Algorithm \ref{alg:0}. These residual norms listed in Table 
\ref{tab:table1}. Notice that the number of linear systems $B(\zeta)x_d=b_d$ 
solved by Algorithm \ref{alg:1} is independent of $N$, and thus the average 
value becomes smaller as $N$ increases. On the other hand, the accuracy 
achieved by Algorithm \ref{alg:1} is also lower. For RSI we 
consider three different dimensions of the starting subspace, set as 
$m=1.1n_{ev},\ m=1.5n_{ev},$ and $m=2n_{ev}$. The rate of convergence of RSI 
is dictated by the ratio $\rho(\lambda_m)/\rho(\lambda_{n_{ev}})$, and thus
the convergence rate improves as $m$ increases. We observe two main trends. 
First, for small values of $m$, e.g., $m=1.1n_{ev}$, RSI is considerably more 
expensive than both Algorithm \ref{alg:0} and Algorithm \ref{alg:1}. This is 
because for small values of $m$ the ratio $\rho(\lambda_m)/\rho(\lambda_{n_{ev}})$ 
might not be close to zero, thus leading to slower convergence in RSI. Second, 
as $m$ increases, the average number of linear systems solved by RSI generally 
decreases, however the large value of $m$ might lead to a few unnecessary solves, 
i.e., for $N=32$ choosing $m=2n_{ev}$ sometimes provides the same accuracy with
$m=1.5n_{ev}$. 

\begin{figure*}[!htbp]
\centering
\includegraphics[width=0.69\textwidth]{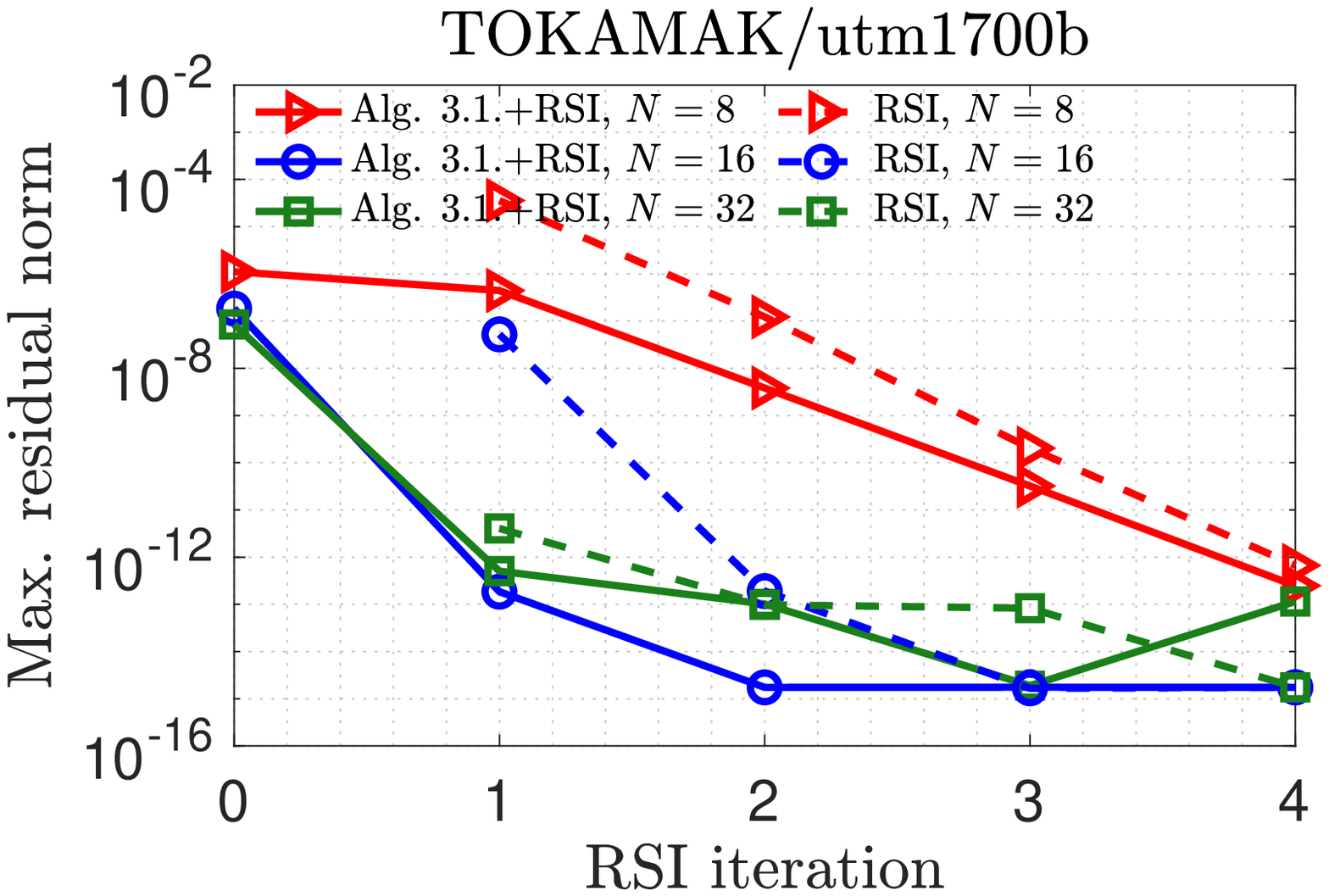}
\includegraphics[width=0.69\textwidth]{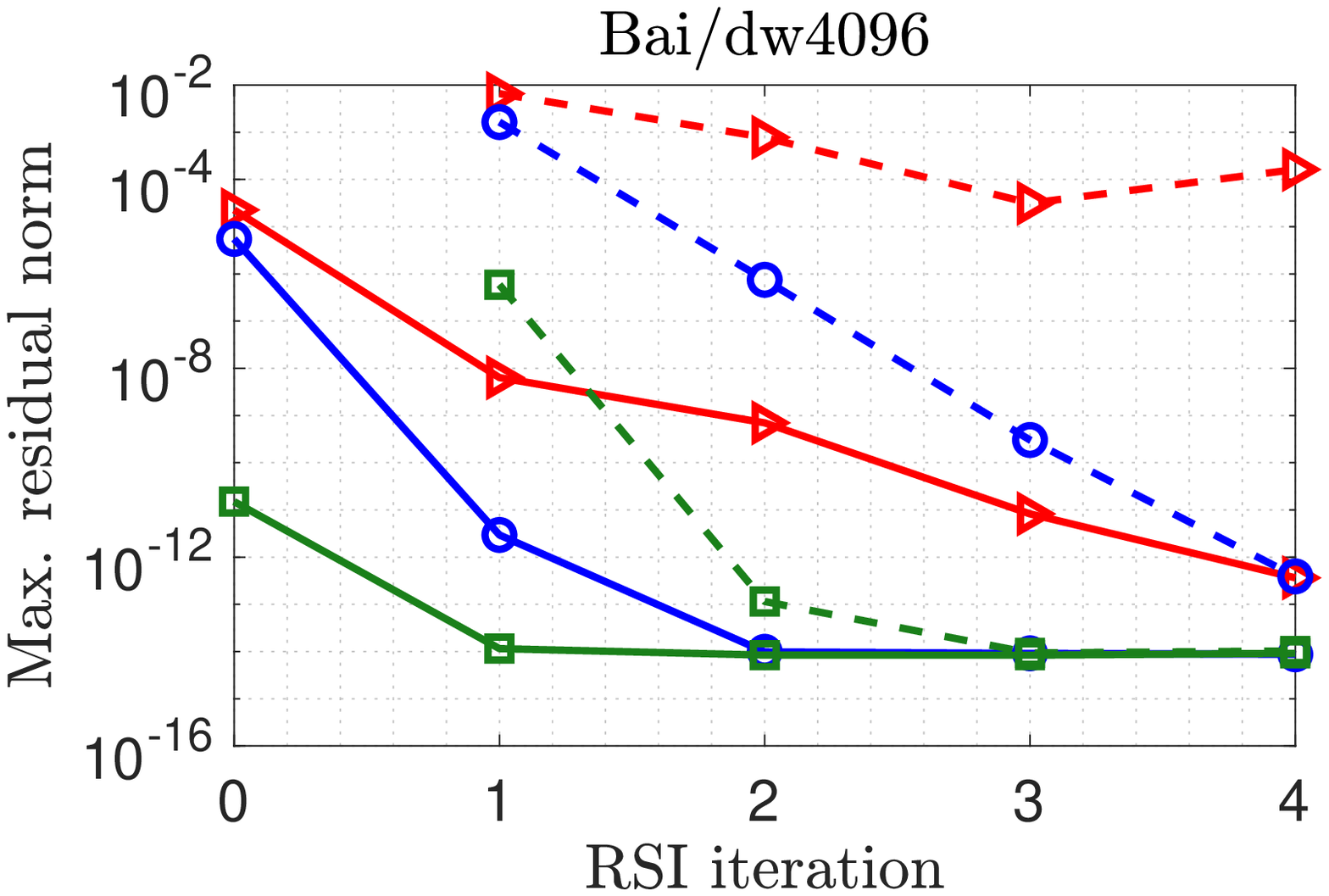}
\caption{{Maximum residual norm achieved at each iteration of 
RSI as $N$ varies and $m=1.5n_{ev},\ n_{ev}=40$. The values listed at the origin 
denote the maximum residual norm achieved by Algorithm \ref{alg:0} before post-processing 
by RSI.}} \label{rsivs}
\end{figure*}
{
In the previous experiment we assumed that the tolerance threshold in RSI was 
dictated by the maximum residual norm achieved by Algorithm \ref{alg:0}. 
Since Algorithm \ref{alg:0} is a one-shot method, its maximum attainable 
accuracy is lower than that of RSI since the latter is an iterative approach. 
Nonetheless, the accuracy of the approximate eigenpairs returned by 
Algorithm \ref{alg:0} can improve by using the corresponding eigenvectors 
as an initial subspace in a separate run of RSI. While this enhancement comes 
at an increased computational cost, it is generally still cheaper than applying 
subspace iteration with a random starting subspace since it avoids the overhead 
associated with the computation of a good approximation of $n_{ev}$ through 
the techniques described in \cite{sakurai2013efficient} and \cite{di2016efficient}. 
Figure \ref{rsivs} plots the maximum residual norm achieved at the end of each 
iteration when RSI is applied to matrices $\texttt{utm1700b}$ and 
$\texttt{dw4096}$ with $m=1.5n_{ev}$ and $n_{ev}=40$. The initial subspace 
in RSI was set using “a)" $m$ random vectors as in the results reported above 
(dashed lines), and “b)" the $n_{ev}$ approximate eigenvectors returned by 
Algorithm \ref{alg:0} augmented by $m-n_{ev}$ random vectors (solid lines). 
In the latter case, the value at the origin denotes 
the accuracy achieved by Algorithm \ref{alg:0} before post-processing by RSI. 
The combination of RSI with Algorithm \ref{alg:0} leads to faster convergence, 
since the initial subspace is of much higher quality. What approach will be 
faster overall depends on the particular problem. For example, choosing 
$N=8$ and a random initial subspace leads to very slow convergence in the case 
of matrix $\texttt{dw4096}$. }

\section{Conclusion} \label{sec5}

This paper presents a class of algorithms for the computation of all eigenvalues 
(and associated eigenvectors) of non-Hermitian matrix pencils located inside a 
disk. The proposed algorithms approximate the sought eigenpairs by harmonic 
Rayleigh-Ritz projections on subspaces built by computing range spaces of rational 
matrix functions through randomized range finders. These rational matrix functions 
are designed so that 
directions associated with non-sought eigenvalues are dampened to (approximately) zero.
Moreover, the proposed algorithms do not require any a priori estimation of the number of 
eigenvalues located inside the disk. The competitiveness of the proposed algorithms 
was demonstrated through numerical experiments performed on a few test problems.

Several research directions are left as future work. One such direction is the extension 
of the algorithms presented in this paper with non-Hermitian Krylov subspace iterative 
solvers and hierarchical preconditioners such as those discussed in \cite{KalantzisSISC22018}. 
Moreover, although this paper focused on algorithms, rational filtering eigenvalue solvers 
owe a large portion of their appeal in the ample parallelism they offer, and a 
distributed memory implementation of the proposed technique would be also of interest. Another 
interesting research direction is the extension of the algorithms presented in this paper 
for the computation of a partial Schur decomposition, or the simultaneous computation of 
both left and right eigenvectors.

\section*{Acknowledgments}

The authors are grateful to the two anonymous referees for their helpful comments 
and suggestions which improved the the readibility and overall quality of the present 
manuscript. The contribution of the first author was partially carried out under 
the support of the Herman H. Goldstine Postdoctoral Fellowship program of the 
International Business Machines Corporation. The work of Yuanzhe Xi was supported 
by NSF grant OAC 2003720. Finally, the authors are indebted to Anthony P. Austin 
for his comments on an earlier draft of the present manuscript.

\bibliographystyle{siam} 
\bibliography{local}

\end{document}